\documentclass[11pt]{article}
\usepackage{amsmath, amssymb, amsthm, graphicx, geometry,xcolor,float, indentfirst, url, hyperref}
\hypersetup{
    citecolor=black,
    colorlinks=true,
    linkcolor=black,
    filecolor=magenta,
    urlcolor=blue,
    pdfpagemode=FullScreen,
    }

\newtheorem{theorem}{Theorem}[section]
\numberwithin{equation}{section}
\newtheorem{prop}[theorem]{Proposition}

\newtheorem{lemma}[theorem]{Lemma}

\newcommand{\E}{\mathbb{E}}
\renewcommand{\P}{\mathbb{P}}
\newcommand{\Var}{\operatorname{Var}}
\newcommand{\Tr}{\operatorname{Tr}}
\newcommand{\R}{\mathbb{R}}
\newcommand{\C}{\mathbb{C}}
\newcommand{\norm}[1]{\left\|#1\right\|}
\newcommand{\abs}[1]{\left|#1\right|}

\title{Spectral Properties of Dense Barab\'asi--Albert Graphs}
\author{Steven J. Miller \\ Williams College \\ \href{mailto:sjm1@williams.edu}{sjm1@williams.edu} \and Arman Rysmakhanov \\ Williams College  \\ \href{mailto:ar21@williams.edu}{ar21@williams.edu}}
\date{}

\begin{document}
\maketitle

\section*{Abstract}
Preferential attachment graphs model networks whose growth produces highly uneven degree distributions, describing many real-world systems. Their adjacency spectra are important because they allow graph-theoretic questions to be studied through the eigenvalues of matrices. We analyze the adjacency matrix of a dense Barab\'asi--Albert (B--A) multigraph, where the number of edges added at each step is proportional to the final number of vertices. First, we compute the large-$n$ limit of the expected adjacency matrix and show that it is described by a rank-one limiting kernel, viewed as a continuous analogue of the adjacency matrix. After centering and scaling, the fluctuations form a random matrix with a computable variance profile. Using the quadratic vector equation approach, we derive the limiting bulk spectral distribution. We also determine the asymptotic location of the leading eigenvalue generated by the rank-one mean component.

\let\thefootnote\relax\footnotetext{This work was supported by NSF Grant DMS2241623, the Finnerty Fund and Williams College.}

\newpage

\tableofcontents

\newpage

\section{Introduction}
\subsection{Background and related work}

Spectral properties of adjacency matrices of graphs, such as eigenvalue distributions and leading eigenvalues/eigenvectors, provide compact summaries of key graph properties, including degree-related information, walk counts, and connectivity patterns \cite{Brouwer2012}.

The Barab\'asi--Albert (B--A) model \cite{BA1999} is a canonical random graph model for networks with heavy-tailed degree distributions. In this model, vertices arrive sequentially and attach to existing vertices with probability proportional to their current degree, producing the well-known scale-free degree distribution observed in many real-world networks, such as the network of scientific collaborations \cite{Barabasi2002}, the functional connections in the human brain \cite{Eguiluz2005}, and the network of Wikipedia hyperlinks \cite{Capocci2006}. In its classical form, the process begins with a complete graph of $m_0$ vertices, and each new vertex attaches to $m \leq m_0$ existing vertices with probability proportional to their degrees, as described in \cite[Section 3.1.1]{Zeng2017ScienceOfScience}.

Spectral properties of scale-free networks are different from classical random matrix behavior. Both empirical and theoretical works show that preferential attachment networks have non-semicircular spectral densities, and their properties are strongly influenced by degree heterogeneity \cite{Dorogovtsev2003SpectraOC, Farkas2001}.

Most rigorous work was done in the sparse regime: each vertex attaches to a small fixed number of existing vertices. In such settings, extreme eigenvalues are strongly influenced by high-degree vertices, with leading eigenvalues often localizing around hubs \cite{Flaxman2005HighDV}. Similar results exist for a wide variety of random trees \cite{Bhamidi2009SpectraOL}. In particular, for sparse preferential-attachment graphs, Montealegre and Vu \cite{Montealegre2018SpectrumOC} describe the limiting bulk spectrum and show localization of the leading eigenvectors.

We work in a qualitatively different, dense preferential attachment regime: the number of edges added at every step is proportional to the network size. In the dense regime, the asymptotic structure of adjacency matrices can be described through the language of graph limit theory. In particular, dense graph sequences can converge to graphon, measurable symmetric kernels that characterize limiting edge densities \cite{LNGL}.

Several works have established convergence of dense multigraphs to deterministic limits \cite{Kolossvry2009MultigraphLA, Rth2009TimeEO, Rth2011MultigraphLO, Borgs2009LimitsOR}. However, deterministic kernel convergence is not sufficient for characterizing the limiting spectrum of adjacency matrices, which also requires the analysis of fluctuations around the mean \cite{Zhu2018AGA}. Typically the limiting spectral distribution is characterized by a self-consistent equation involving the resolvent of the adjacency matrix. The quadratic vector equation (QVE) approach developed by Ajanki, Erd\H{o}s, and Kr\"uger provides one general method for finding limiting spectral densities of such Wigner-type matrices with arbitrary variance profiles under conditions \cite{AEK2017}. We use the graphon approach of Zhu \cite{Zhu2018AGA} that builds on the QVE theory, which characterizes the limiting measure through a continuum QVE. Outlier behavior is not captured by the limiting spectrum, though. For that end, we use finite-rank perturbation theory to show that adding a low-rank deterministic component to a random matrix can produce isolated eigenvalues separated from the bulk spectrum. Such outliers were studied in spiked random matrix models \cite{BenaychGeorges2009TheEA} and in inhomogeneous Erd\H{o}s--R\'enyi graphs with finite-rank kernels \cite{ChakrabartyChakrabortyHazra2020}.

Motivated by these developments, we study the spectral properties of dense preferential attachment multigraphs in a regime where the number of attachments grows linearly with the network size. We show that the expected adjacency matrix converges to an explicit kernel that generates a rank-one component. We compare the centered and scaled adjacency matrix with an independent Poisson surrogate matrix that has an explicit limiting variance kernel. Zhu's graphon convergence theorem \cite{Zhu2018AGA} characterizes the limiting bulk spectral measure through a QVE, which reduces to a system of two scalar equations. This method allows us to obtain a deterministic limiting bulk spectral distribution and an explicit asymptotic expression for the top eigenvalue. This work complements existing findings on the spectra of sparse preferential attachment graphs, conecting connect dense preferential attachment regime limits to modern random matrix theory.

\subsection{Model and notation}

Let $n$ be the final number of vertices. Fix rational $c \in \left(0, 1 / 2\right)$, and consider only values of $n$ for which $m := c n$ is an integer. All limits are taken as $n \to \infty$ along this admissible sequence. Time $t$ indexes vertices: vertex $t$ arrives at time $t$, except for the first $m + 1$ vertices.
\newline

\noindent\textbf{Initialization:} At time $t_0 := m + 1$, the graph is a simple clique $K_{m + 1}$, a fully connected graph of $m + 1$ vertices. Let $E_t$ be the number of edges at time $t$, and $D_t := 2 E_t$ the total degree.
\newline

\noindent\textbf{Preferential attachment:} For each $t > t_0$ such that $t \le n$, we add a new vertex to the graph -- a total of $n - m - 1$ new vertices added to the initial clique. In the limit, the number of new vertices is on the order of $m$ since $n - m - 1 = \frac{1 - c}{c} m - 1$. Each new vertex $t$ adds $m$ edges independently to the existing vertices $\left\{1, \dots, t - 1\right\}$; the probability that a single edge connects vertex $t$ to a vertex $j < t$ is
\[
q_{t \to j}
\ = \
\frac{d_j \left(t - 1\right)}{\sum_{v = 1}^{t - 1} d_v \left(t - 1\right)}
\ = \
\frac{d_j \left(t - 1\right)}{D_{t - 1}},
\]
where $d_j \left(s\right)$ is the degree of $j$ at time $s$. Multiple edges for the same pair of vertices are allowed.
\newline

\noindent\textbf{Total degree:} At initialization, $E_{t_0} = \binom{m + 1}{2}$, hence
\[
D_{t_0}
\ = \
m \left(m + 1\right).
\]
Each step $t > t_0$ adds exactly $m$ edges, so
\begin{equation}\label{eq:Dt}
D_t
\ = \
D_{t_0} + 2 m \left(t - t_0\right)
\ = \
2 m t - m \left(m + 1\right)
\end{equation}
for all $t \ge t_0$. We write
\[
\mathbb H
\ := \
\left\{z \in \C : \Im z > 0\right\},
\]
and let $\mathcal F_t$ denote the sigma-algebra generated by the graph process through time $t$.

\subsection{Main results}

\begin{theorem}[Limiting bulk spectrum]\label{thm:bulk-spectrum}
Let $M$ be the adjacency matrix of the Barab\'asi--Albert multigraph with $m = c n$ and fixed rational $c \in \left(0, 1 / 2\right)$, where $n$ is taken to $\infty$ along values for which $m$ is an integer, and define the centered and rescaled matrix
\[
Y
\ := \
\frac{M - \E\left[M\right]}{\sqrt{m}}.
\]
Then the empirical spectral distribution of $Y$ converges weakly in probability,\footnote{A sequence of random probability measures $\mu_n$ on $\R$ converges weakly in probability to a probability measure $\mu$ if, for all bounded continuous functions $f : \R \to \R$ and all $\varepsilon > 0$,
\[
\P \left[\abs{\int f \, d \mu_n - \int f \, d \mu} \ > \ \varepsilon\right]
\ \to \
0
\]
as $n \to \infty$.} as $n \to \infty$, to a deterministic probability measure $\mu_c$.

The Stieltjes transform
\[
G_c \left(z\right)
\ := \
\int_{\R} \frac{1}{\lambda - z} \, d \mu_c \left(\lambda\right),
\qquad z \in \mathbb{H},
\]
equals
\[
G_c \left(z\right)
\ = \
c m_A \left(z\right) - \int_c^1 \frac{1}{z + u \left(x\right) K \left(z\right)} \, d x,
\qquad
u \left(x\right)
\ = \
\left(2 x - c\right)^{-1 / 2},
\]
where the pair $\left(K \left(z\right), L \left(z\right)\right)$ in $\mathbb{H}^2$ is the unique solution of
\begin{equation}\label{eq:QVE-system}
\begin{aligned}
K \left(z\right)
&\ = \
L \left(z\right) - \frac{\sqrt{c}}{z + L \left(z\right) / \sqrt{c}},\\
L \left(z\right)
&\ = \
- \int_c^1 \frac{u \left(y\right)}{z + u \left(y\right) K \left(z\right)} \, d y,
\end{aligned}
\qquad z \in \mathbb{H},
\end{equation}
and
\[
m_A \left(z\right)
\ = \
- \frac{1}{z + L \left(z\right) / \sqrt{c}}.
\]
\end{theorem}

\begin{theorem}[Leading eigenvalue]\label{thm:top-eigenvalue}
Let $M$ be the adjacency matrix of the Barab\'asi--Albert multigraph with $m = c n$ and fixed rational $c \in \left(0, 1 / 2\right)$ such that $n$ and $m$ are integers. Then the largest eigenvalue of the scaled adjacency matrix satisfies
\[
\lambda_{\max} \left(\frac{M}{\sqrt{m}}\right)
\ = \
\frac{\alpha \left(c\right)}{\sqrt{c}} \sqrt{n} + O_{\P} \left(n^{-1 / 2}\right),\footnote{For positive sequences $a_n$, the notation $X_n = O_{\P} \left(a_n\right)$ means that the sequence $X_n / a_n$ is bounded in probability: for every $\varepsilon > 0$, there exist constants $C_\varepsilon < \infty$ and $N_\varepsilon$ such that
\[
\P \left(\abs{X_n} > C_\varepsilon a_n\right)
\ < \
\varepsilon
\]
for all $n \ge N_\varepsilon$. Similarly, $X_n = o_{\P} \left(a_n\right)$ means $X_n / a_n \to 0$ in probability.}
\qquad n \to \infty,
\]
with
\[
\alpha \left(c\right)
\ = \
c + \frac{c}{2} \log \frac{2 - c}{c}.
\]
\end{theorem}

See Figure 1 for an empirical verification of the bulk law obtained from numerically inverting the Stieltjes transform from Theorem \ref{thm:bulk-spectrum} and the outlier eigenvalue asymptotics from Theorem \ref{thm:top-eigenvalue}.

\begin{figure}[H]
    \centering
    \includegraphics[width=1\linewidth]{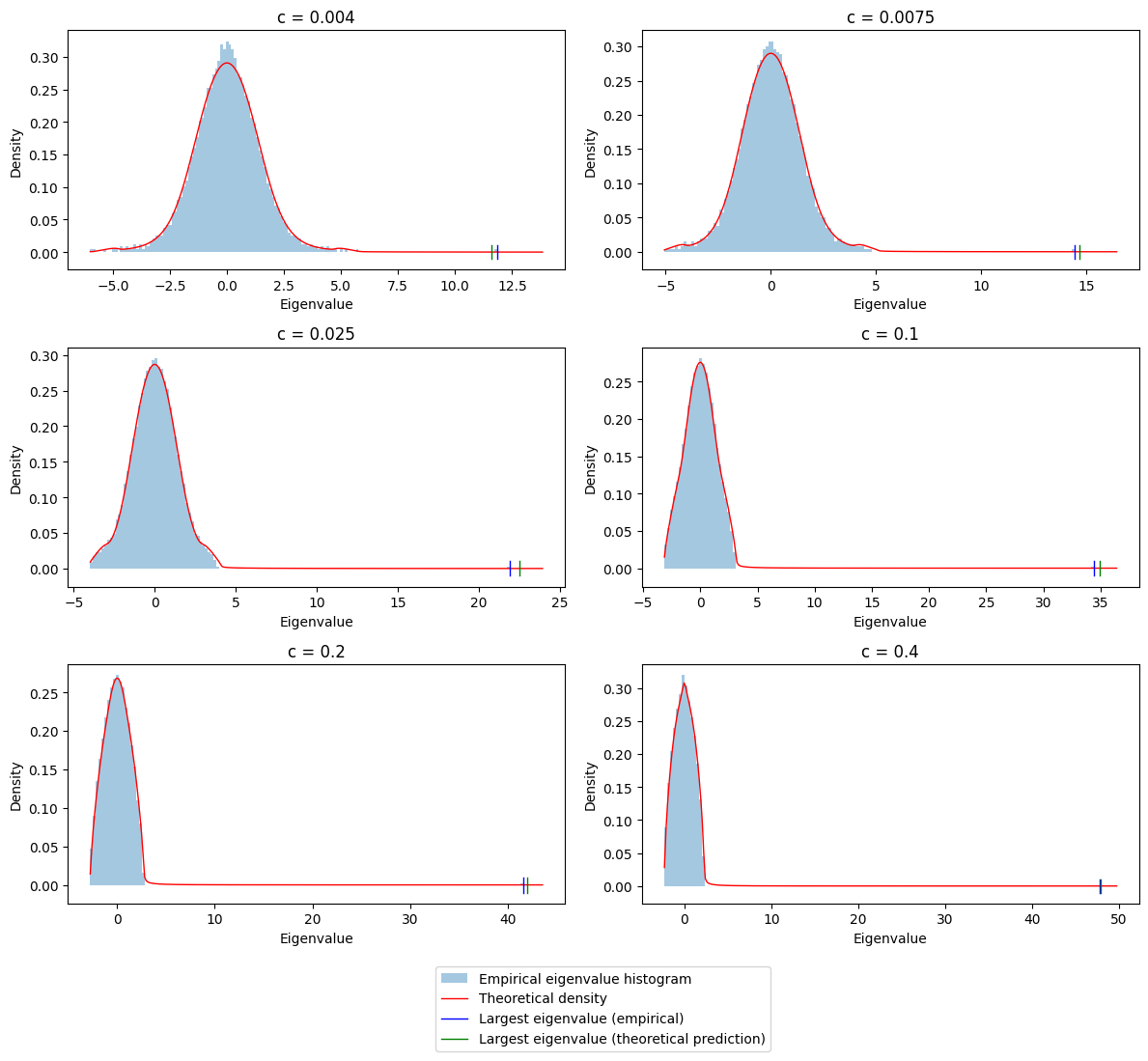}
    \caption{Spectrum of random B--A multigraphs for several values of $c$. The centered empirical spectral distributions closely match the value obtained from the limiting Stieltjes transform. As $c$ increases, the bulk becomes more concentrated near the origin. The shapes of the densities match previous results \cite[Figure 4]{Farkas2001}, which report triangle-like central parts and decaying tails for sparse fixed-$m$ regime.}
\end{figure}

\section{Finding the expected adjacency matrix.}

\subsection{Expected degrees of individual vertices.}

Fix a vertex $j$ and $t > \max\left\{j, t_0\right\}$. Define the random variable $\Delta d_j \left(t\right)$ to be the number of new edges from $t$ to $j$ at time $t$. Then, conditional on the state of the graph at time $t - 1$, $\mathcal{F}_{t - 1}$, we have
\[
\Delta d_j \left(t\right) \,\big|\, \mathcal{F}_{t - 1}
\ \sim \
\mathrm{Binomial} \left(m, q_{t \to j}\right),
\qquad
\E\left[\Delta d_j \left(t\right) \,\big|\, \mathcal{F}_{t - 1}\right]
\ = \
m q_{t \to j}
\ = \
m \frac{d_j \left(t - 1\right)}{D_{t - 1}}.
\]
Using~\eqref{eq:Dt}, this gives a linear recursion for $\mu_j \left(t\right) := \E\left[d_j \left(t\right)\right]$. For $t > \max\left\{j, t_0\right\}$,
\begin{equation}\label{eq:mu-rec}
\mu_j \left(t\right)
\ = \
\mu_j \left(t - 1\right) \left(1 + \frac{m}{2 m \left(t - 1\right) - m \left(m + 1\right)}\right)
\ = \
\mu_j \left(t - 1\right) \left(1 + \frac{1}{2 \left(t - 1\right) - \left(m + 1\right)}\right).
\end{equation}
At the time of its arrival, each vertex that is not in initial clique has degree $m$, so $\mu_j \left(j\right) = m$.

Thus, with $s_0 := \max\left\{j, t_0\right\}$,
\begin{equation}\label{eq:mu-product}
\mu_j \left(t\right)
\ = \
m \prod_{s = s_0 + 1}^{t} \left(1 + \frac{1}{2 \left(s - 1\right) - \left(m + 1\right)}\right)
\ = \
m \prod_{s = s_0 + 1}^{t} \frac{s - \frac{m + 2}{2}}{s - \frac{m + 3}{2}}.
\end{equation}

This product can be represented using Gamma-functions,
\begin{equation}\label{eq:mu-gamma}
\mu_j \left(t\right)
\ = \
m \frac{\Gamma \left(t - \frac{m}{2}\right)}{\Gamma \left(s_0 - \frac{m}{2}\right)} \frac{\Gamma \left(s_0 - \frac{m + 1}{2}\right)}{\Gamma \left(t - \frac{m + 1}{2}\right)},
\qquad t \ \ge \ s_0,
\end{equation}
which follows from Lemma \ref{lem:prod-gamma}.

\begin{lemma}\label{lem:prod-gamma}
Let $r_0 \le r_1$ be integers and $a, b \in \C$ such that the Gamma terms below are defined. The product is empty when $r_0 = r_1$ and is defined to be $1$. Then
\[
\prod_{r = r_0 + 1}^{r_1} \frac{r - a}{r - b}
\ = \
\frac{\Gamma \left(r_1 + 1 - a\right)}{\Gamma \left(r_0 + 1 - a\right)} \frac{\Gamma \left(r_0 + 1 - b\right)}{\Gamma \left(r_1 + 1 - b\right)}.
\]
\end{lemma}

\begin{proof}
Using the standard result $\Gamma \left(z + 1\right) = z \Gamma \left(z\right)$,
\begin{align*}
\Gamma \left(r_1 + 1 - a\right)
&\ = \
\left(r_1 - a\right) \Gamma \left(r_1 - a\right) \\
&\ = \
\left(r_1 - a\right) \left(r_1 - 1 - a\right) \Gamma \left(r_1 - 1 - a\right) \\
&\ \ \, \vdots \\
&\ = \
\left(\prod_{r = r_0 + 1}^{r_1} \left(r - a\right)\right) \Gamma \left(r_0 + 1 - a\right).
\end{align*}

Hence
\[
\prod_{r = r_0 + 1}^{r_1} \left(r - a\right)
\ = \
\frac{\Gamma \left(r_1 + 1 - a\right)}{\Gamma \left(r_0 + 1 - a\right)}.
\]
The same argument applies to $b$ and taking their ratio proves the lemma.
\end{proof}

\subsection{Expectation of the edge count $\E\left[M\right]$}

For $k > t_0$ and $j < k$, all edges between $k$ and $j$ are created at time $k$. Conditional on $\mathcal F_{k - 1}$, we have
\[
M_{k j} \,\big|\, \mathcal{F}_{k - 1}
\ \sim \
\mathrm{Binomial} \left(m, \frac{d_j \left(k - 1\right)}{D_{k - 1}}\right),
\]
since in the multigraph model, the $m$ edges from $k$ are placed independently, each choosing $j$ with probability $d_j \left(k - 1\right) / D_{k - 1}$, making $M_{k j}$ a binomial random variable.

Hence, using \eqref{eq:Dt} and \eqref{eq:mu-gamma},
\begin{equation}\label{eq:E-Mij-discrete}
\E\left[M_{k j}\right]
\ = \
m \frac{\mu_j \left(k - 1\right)}{D_{k - 1}}
\ = \
\frac{m^2}{2 m \left(k - 1\right) - m \left(m + 1\right)} \frac{\Gamma \left(k - 1 - \frac{m}{2}\right)}{\Gamma \left(s_0 - \frac{m}{2}\right)} \frac{\Gamma \left(s_0 - \frac{m + 1}{2}\right)}{\Gamma \left(k - 1 - \frac{m + 1}{2}\right)},
\end{equation}
where $s_0 := \max\left\{j, t_0\right\}$ and $t_0 = m + 1$ is the initialization time.

For $1 \le k, j \le t_0$, the initial clique $K_{m + 1}$ yields $\E\left[M_{k j}\right] = \mathbf{1}_{k \ne j}$ and no later edges affect those pairs.

\section{Kernel Approach to $\E\left[M\right]$}

\subsection{Definition of a kernel}

Following Lov\'asz \cite{LNGL}, we define a \textbf{kernel} to be any bounded, symmetric, measurable function
\[
W : \left[0, 1\right]^2
\ \to \
\R,
\]
which we can use to generalize adjacency matrices of graphs.

Given a finite graph $H$ with adjacency matrix $A = \left(A_{k j}\right)$, the associated step-function kernel $W_H$ is defined by partitioning $\left[0, 1\right]$ into $\abs{V \left(H\right)}$ equal intervals $I_1, \dots, I_{\abs{V \left(H\right)}}$ and setting $W_H \left(x, y\right) = A_{k j}$ whenever $x \in I_k, \ y \in I_j$.

A \textbf{graphon} is a kernel whose values lie in the interval $\left[0, 1\right]$, corresponding to simple graphs \cite{LNGL}. In general, kernels serve as the limit objects for sequences of weighted or multigraph adjacency matrices.

\subsection{The kernel of $\E\left[M\right]$}

Now we calculate the kernel for the expected adjacency matrix of dense B--A multigraphs.

\begin{theorem}\label{thm:mean-kernel-limit}
Let $c \in \left(0, 1 / 2\right)$ and $m = c n$. For indices $k > t_0$ and $j < k$, with $k = \lfloor y n \rfloor$ and $j = \lfloor x n \rfloor$ for $x \in \left(0, y\right)$ and $y \in \left(c, 1\right]$, set $t_0 = m + 1$ and $s_0 = \max\left\{j, t_0\right\}$. Then
\[
\E\left[M_{k j}\right]
\ = \
m \frac{\mu_j \left(k - 1\right)}{D_{k - 1}}
\ \to \
\frac{c}{\sqrt{\left(2 \max\left\{x, c\right\} - c\right) \left(2 y - c\right)}},
\qquad n \to \infty,
\]
where $D_t = 2 m t - m \left(m + 1\right)$ and $\mu_j \left(\cdot\right) = \E\left[d_j \left(\cdot\right)\right]$. Moreover, there exists a constant $C > 0$, depending only on $c$, such that
\[
\abs{\E\left[M_{k j}\right] - F \left(\frac{k}{n}, \frac{j}{n}\right)}
\ \le \
\frac{C}{n}
\]
for all $k > t_0$ and $j < k$. In particular, $\E\left[M_{k j}\right] = O \left(1\right)$ uniformly in $j, k, n$.
\end{theorem}

\begin{proof}
We first prove the uniform bound. Set
\[
s_0
\ := \
\max\left\{j, t_0\right\}.
\]
From the degree evolution, for $t > s_0$,
\[
\E\left[d_j \left(t\right) \mid \mathcal F_{t - 1}\right]
\ = \
\left(1 + \frac{m}{D_{t - 1}}\right) d_j \left(t - 1\right).
\]
Taking expectations gives
\[
\mu_j \left(t\right)
\ = \
\E\left[\E\left[d_j \left(t\right) \mid \mathcal F_{t - 1}\right]\right]
\ = \
\left(1 + \frac{m}{D_{t - 1}}\right) \mu_j \left(t - 1\right).
\]
Then we get
\begin{align*}
\mu_j \left(t\right)
&\ = \
\E_{\mathcal F_t}\left[d_j \left(t\right)\right] \\
&\ = \
\E_{\mathcal F_{t - 1}}\left[\E_{\Delta d_j \left(t\right)}\left[d_j \left(t\right) \mid \mathcal F_{t - 1}\right]\right] \\
&\ = \
\E_{\mathcal F_{t - 1}}\left[\left(1 + \frac{m}{D_{t - 1}}\right) d_j \left(t - 1\right)\right] \\
&\ = \
\left(1 + \frac{m}{D_{t - 1}}\right) \E_{\mathcal F_{t - 1}}\left[d_j \left(t - 1\right)\right] \\
&\ = \
\left(1 + \frac{m}{D_{t - 1}}\right) \mu_j \left(t - 1\right).
\end{align*}
Since
\[
D_{t - 1}
\ \ge \
D_{t_0}
\ = \
m \left(m + 1\right),
\]
we have
\[
\frac{m}{D_{t - 1}}
\ \le \
\frac{1}{m + 1}
\ \le \
\frac{C_1}{n}
\]
for some constant $C_1 > 0$. Since $\mu_j \left(s_0\right) = m$, iterating gives
\[
\mu_j \left(k - 1\right)
\ \le \
m \left(1 + \frac{C_1}{n}\right)^{k - 1 - s_0}
\ \le \
m e^{C_1}
\ = \
C_2 n,
\qquad
C_2
\ = \
c e^{C_1}.
\]
If $k - 1 = s_0$, the exponent is zero and the first inequality is an equality. Therefore
\[
\E\left[M_{k j}\right]
\ = \
m \frac{\mu_j \left(k - 1\right)}{D_{k - 1}}
\ \le \
m \frac{C_2 n}{m \left(m + 1\right)}
\ \le \
C,
\]
which proves the uniform bound.

From the Gamma-function representation \eqref{eq:mu-gamma},
\[
\frac{\mu_j \left(k - 1\right)}{n}
\ = \
\frac{m}{n} \cdot \frac{\Gamma \left(k - 1 - \frac{m}{2}\right)}{\Gamma \left(k - 1 - \frac{m + 1}{2}\right)} \cdot \frac{\Gamma \left(s_0 - \frac{m + 1}{2}\right)}{\Gamma \left(s_0 - \frac{m}{2}\right)}.
\]
Set
\[
z_0
\ := \
s_0 - \frac{m + 1}{2},
\qquad
z_1
\ := \
k - 1 - \frac{m + 1}{2}.
\]
Since $s_0 \ge t_0 = m + 1$ and $k \ge m + 2$, we have
\[
z_0
\ \ge \
\frac{m + 1}{2},
\qquad
z_1
\ \ge \
\frac{m + 1}{2}.
\]
Thus, since $m = c n$,
\[
z_0, z_1
\ \ge \
\frac{c}{2} n
\]
for all sufficiently large $n$.

For fixed $a, b \in \C$, the Gamma-ratio asymptotic \cite{DLMF_Gamma} gives
\[
\frac{\Gamma \left(z + a\right)}{\Gamma \left(z + b\right)}
\ = \
z^{a - b} \left(1 + O \left(\frac{1}{z}\right)\right),
\qquad
z \to \infty.
\]
Since $z_0, z_1 \ge c n / 2$, the corresponding error estimates are uniform in $k > t_0$ and $j < k$. Hence
\[
\frac{\Gamma \left(z_1 + \frac{1}{2}\right)}{\Gamma \left(z_1\right)}
\ = \
z_1^{1 / 2} \left(1 + O \left(n^{-1}\right)\right),
\]
and
\[
\frac{\Gamma \left(z_0\right)}{\Gamma \left(z_0 + \frac{1}{2}\right)}
\ = \
z_0^{-1 / 2} \left(1 + O \left(n^{-1}\right)\right),
\]
uniformly in $k > t_0$ and $j < k$.

Therefore
\[
\frac{\mu_j \left(k - 1\right)}{n}
\ = \
c \sqrt{\frac{z_1}{z_0}} \left(1 + O \left(n^{-1}\right)\right),
\]
uniformly in $k > t_0$ and $j < k$.

Next, from $D_t = 2 m t - m \left(m + 1\right)$,
\[
D_{k - 1}
\ = \
2 m \left(k - 1\right) - m \left(m + 1\right)
\ = \
2 m z_1.
\]
Hence
\begin{align*}
\E\left[M_{k j}\right]
&\ = \
m \frac{\mu_j \left(k - 1\right)}{D_{k - 1}} \\
&\ = \
\frac{c}{\sqrt{\left(\frac{2 z_0}{n}\right) \left(\frac{2 z_1}{n}\right)}} \left(1 + O \left(n^{-1}\right)\right).
\end{align*}

Now
\[
\frac{2 z_1}{n}
\ = \
2 \frac{k}{n} - c - \frac{3}{n},
\]
so
\[
\frac{2 z_1}{n}
\ = \
2 \frac{k}{n} - c + O \left(n^{-1}\right)
\]
uniformly in $k$.

Also,
\[
\frac{2 z_0}{n}
\ = \
2 \frac{s_0}{n} - c - \frac{1}{n}.
\]
Since $s_0 = \max\left\{j, m + 1\right\}$,
\[
\frac{2 z_0}{n}
\ = \
2 \max\left\{\frac{j}{n}, c\right\} - c + O \left(n^{-1}\right)
\]
uniformly in $j$.

The quantities
\[
2 \max\left\{\frac{j}{n}, c\right\} - c,
\qquad
2 \frac{k}{n} - c
\]
are both bounded below by $c$. Therefore the function
\[
\left(a, b\right)
\ \mapsto \
\frac{c}{\sqrt{a b}}
\]
has uniformly bounded first derivatives on the relevant range. It follows that
\[
\E\left[M_{k j}\right]
\ = \
\frac{c}{\sqrt{\left(2 \max\left\{\frac{j}{n}, c\right\} - c\right) \left(2 \frac{k}{n} - c\right)}} + O \left(n^{-1}\right)
\]
uniformly in $k > t_0$ and $j < k$.

In particular, if
\[
k
\ = \
\lfloor y n \rfloor,
\qquad
j
\ = \
\lfloor x n \rfloor,
\]
with $x \in \left(0, y\right)$ and $y \in \left(c, 1\right]$, then
\[
\E\left[M_{k j}\right]
\ \to \
\frac{c}{\sqrt{\left(2 \max\left\{x, c\right\} - c\right) \left(2 y - c\right)}}.
\]

Solving for different cases, $A := \left[0, c\right]$ and $B := \left(c, 1\right]$, results in $F : \left[0, 1\right]^2 \to \left[0, \infty\right)$, which corresponds to the expected edge count for a given vertex pair in the limit:
\begin{equation}\label{eq:FWR}
F \left(x, y\right)
\ := \
\lim_{n \to \infty} \E\left[M_{\max\left\{1, \lfloor n x \rfloor\right\}, \max\left\{1, \lfloor n y \rfloor\right\}}\right]
\ = \
\begin{cases}
1, & x, y \in A, \ x \ \ne \ y,\\
\displaystyle \sqrt{\frac{c}{2 y - c}}, & x \in A, \ y \in B, \ x \ \ne \ y,\\
\displaystyle \sqrt{\frac{c}{2 x - c}}, & x \in B, \ y \in A, \ x \ \ne \ y,\\
\displaystyle \frac{c}{\sqrt{\left(2 x - c\right) \left(2 y - c\right)}}, & x, y \in B, \ x \ \ne \ y,\\
0, & x \ = \ y;
\end{cases}
\end{equation}
see Figure 2.

Consequently, for all $k > t_0$ and $j < k$,
\[
\E\left[M_{k j}\right]
\ = \
F \left(\frac{k}{n}, \frac{j}{n}\right) + O \left(n^{-1}\right),
\]
where the error is uniform in $j, k, n$.
\end{proof}

\begin{figure}[H]
    \centering
    \includegraphics[width=0.75\linewidth]{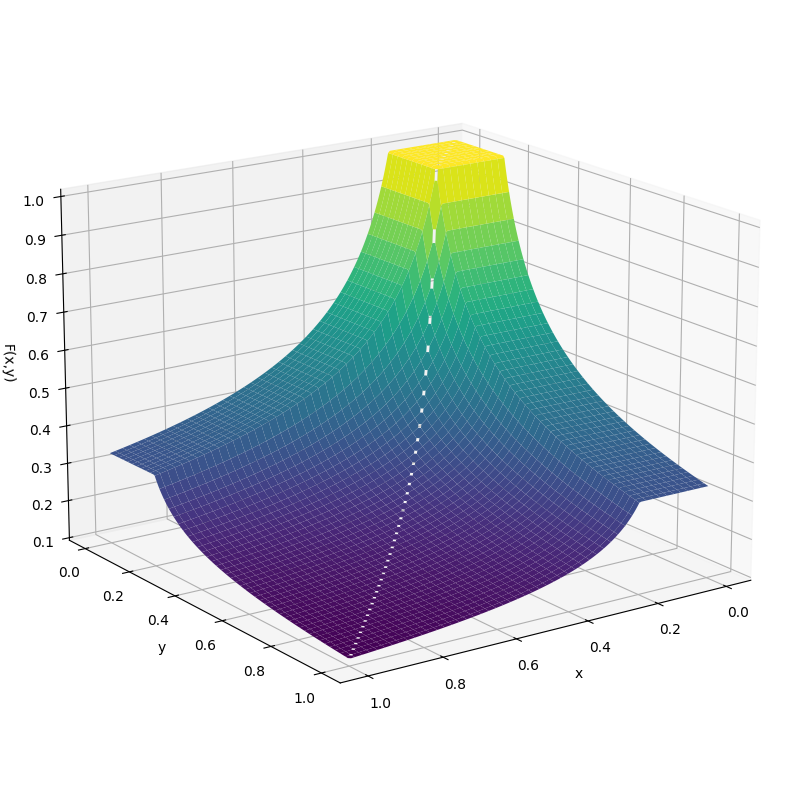}
    \caption{Plot of $F \left(x, y\right)$.}
\end{figure}

\section{Bulk law}

\subsection{Centering and isolation of the bulk}

\paragraph{Definition.}\label{def:centering}
Let $M$ be the $\left(\text{symmetric}\right)$ adjacency matrix of the B--A multigraph and set
\[
Y
\ := \
\frac{M - \E\left[M\right]}{\sqrt{m}},
\qquad
m
\ = \
c n,
\qquad
0 \ < \ c \ < \ \frac{1}{2}.
\]

We use $Y$ to focus on the bulk of the spectrum.

\begin{lemma}[Rank--one structure of the limiting mean kernel]\label{lem:EM-rank1}
Let $A = \left[0, c\right]$, $B = \left(c, 1\right]$, and define $u \left(t\right) = \left(2 t - c\right)^{-1 / 2}$ for $t \in B$. Define $\phi : \left[0, 1\right] \to \R$ by
\[
\phi \left(x\right)
\ := \
\begin{cases}
1, & x \in A,\\
\sqrt{c} u \left(x\right), & x \in B.
\end{cases}
\]
Then the limiting mean kernel $F$ from \eqref{eq:FWR} satisfies
\[
F \left(x, y\right)
\ = \
\phi \left(x\right) \phi \left(y\right),
\qquad
x \ \ne \ y.
\]
Thus, apart from its diagonal values, $F$ has rank one.
\end{lemma}

\begin{proof}
By definition of $F$ and of $\phi$,
\[
\phi \left(x\right) \phi \left(y\right)
\ = \
\begin{cases}
1, & x, y \in A, \\
\sqrt{c} u \left(y\right), & x \in A, \ y \in B, \\
\sqrt{c} u \left(x\right), & x \in B, \ y \in A, \\
c u \left(x\right) u \left(y\right), & x, y \in B,
\end{cases}
\]
which matches $F \left(x, y\right)$ case-by-case whenever $x \ne y$. Hence
\[
F \left(x, y\right)
\ = \
\phi \left(x\right) \phi \left(y\right),
\qquad
x \ \ne \ y.
\]

We also have $F \left(x, x\right) = 0$ by definition. Since
\[
\left\{\left(x, x\right) : x \in \left[0, 1\right]\right\}
\]
has Lebesgue measure zero in the unit square, the diagonal values do not affect the limiting kernel, so $F$ and the rank--one kernel $\phi \left(x\right) \phi \left(y\right)$ agree almost everywhere.
\end{proof}

A rank-one desciption of the expected adjacency matrix of a variant of random B--A graph also appears in \cite{Bianconi2002MeanFS}.

\begin{lemma}[Centering preserves the bulk spectrum]
\label{lem:centering-bulk}
The empirical spectral distributions of $Y$ and $M / \sqrt{m}$ have the same weak limit in probability, whenever either converges to a deterministic probability measure.
\end{lemma}

\begin{proof}
Define $v \in \R^n$ componentwise as $v_i = \phi \left(i / n\right)$, with $\phi$ given in Lemma \ref{lem:EM-rank1}. Set $R := \E\left[M\right] - v v^\top$. Using Theorem \ref{thm:mean-kernel-limit} and Lemma \ref{lem:EM-rank1}, $\abs{R_{i j}} \le C / n$ for $i \ne j$ when $\max\left\{i, j\right\} > t_0$, where $C$ is independent of $i, j, n$. For the initial clique, $\E\left[M_{i j}\right] = 1$ for $i \ne j$. We have $v_i = 1$ for $i \le m$ and
\[
v_{m + 1}
\ = \
\left(1 + 2 / m\right)^{-1 / 2},
\qquad
0
\ \le \
1 - v_{m + 1}
\ \le \
\frac{1}{m}
\ = \
\frac{1}{c n}.
\]
Thus, there exists a constant $C$ such that $\abs{R_{i j}} \le C / n$ for all $i \ne j$. Self-loops are not allowed, so $M_{i i} = 0$. Combining that with $0 < v_i \le 1$, we get $\abs{R_{i i}} = v_i^2 \le 1$. Since $R$ is real symmetric, its operator norm is the largest absolute value of its eigenvalues. By Gershgorin's theorem \cite[Theorem 6.1.1]{HornJohnson2013}, each eigenvalue $\lambda$ of $R$ satisfies
\[
\abs{\lambda - R_{i i}}
\ \le \
\sum_{j \ne i} \abs{R_{i j}}
\]
for some $i$. By the triangle inequality,
\[
\abs{\lambda}
\ \le \
\abs{R_{i i}} + \sum_{j \ne i} \abs{R_{i j}}
\ \le \
1 + \frac{\left(n - 1\right) C}{n}
\ \le \
1 + C.
\]
Hence $\norm{R} \le 1 + C$.

Now write
\[
A
\ := \
Y + \frac{R}{\sqrt{m}},
\qquad
B
\ := \
\frac{M}{\sqrt{m}}
\ = \
A + \frac{v v^\top}{\sqrt{m}},
\qquad
\delta_n
\ := \
\frac{1 + C}{\sqrt{m}}.
\]
where $\delta_n \to 0$.

Let $\lambda_1 \left(X\right) \ge \cdots \ge \lambda_n \left(X\right)$ be eigenvalues of a symmetric $n \times n$ matrix $X$. Define their empirical distribution function as
\[
F_X \left(t\right)
\ := \
\frac{1}{n} \# \left\{j : \lambda_j \left(X\right) \ \le \ t\right\}.
\]
We have $A - Y = R / \sqrt{m}$. Weyl's perturbation theorem \cite[Corollary III.2.6]{BhatiaMatrixAnalysis} gives
\[
\max_{1 \le j \le n} \abs{\lambda_j \left(A\right) - \lambda_j \left(Y\right)}
\ \le \
\norm{A - Y}
\ = \
\frac{\norm{R}}{\sqrt{m}}
\ \le \
\delta_n.
\]
which means
\[
F_Y \left(t - \delta_n\right)
\ \le \
F_A \left(t\right)
\ \le \
F_Y \left(t + \delta_n\right).
\]

Since $B - A = v v^\top / \sqrt{m}$ is positive semidefinite of rank one, the interlacing inequality \cite[Exercise III.2.4]{BhatiaMatrixAnalysis} gives
\[
\lambda_1 \left(B\right)
\ \ge \
\lambda_1 \left(A\right)
\ \ge \
\lambda_2 \left(B\right)
\ \ge \
\lambda_2 \left(A\right)
\ \ge \
\cdots
\ \ge \
\lambda_n \left(B\right)
\ \ge \
\lambda_n \left(A\right).
\]
Hence
\[
F_A \left(t\right) - \frac{1}{n}
\ \le \
F_B \left(t\right)
\ \le \
F_A \left(t\right).
\]
This and the previous bound result in
\[
F_Y \left(t - \delta_n\right) - \frac{1}{n}
\ \le \
F_B \left(t\right)
\ \le \
F_Y \left(t + \delta_n\right),
\]
and similarly,
\[
F_B \left(t - \delta_n\right)
\ \le \
F_Y \left(t\right)
\ \le \
F_B \left(t + \delta_n\right) + \frac{1}{n}.
\]

Define the Lévy distance for distribution functions $F, G$ as
\[
d_L \left(F, G\right)
\ := \
\inf\left\{\varepsilon \ > \ 0 : F \left(t - \varepsilon\right) - \varepsilon \ \le \ G \left(t\right) \ \le \ F \left(t + \varepsilon\right) + \varepsilon, \ \forall t \in \R\right\}.
\]
The preceding bounds give
\[
d_L \left(F_B, F_Y\right)
\ \le \
\max\left\{\delta_n, 1 / n\right\}
\ \to \
0.
\]
By \cite[Theorem 6]{GibbsSu2002}, the L\'evy distance metrizes weak convergence. So the the triangle inequality implies that $B$ and $Y$ have the same limiting spectral distribution in probability, when either limit exists.
\end{proof}

\subsection{Poissonization}

\indent Now we use standard results (Lemma \ref{lem:multi-poisson}, Propositions \ref{prop:perturb-bound} and \ref{prop:poisson-dev}) to construct a matrix $\widetilde M$ sequentially so that, for each $k > t_0$, conditional on the history up to time $k - 1$, the entries $\widetilde M_{k 1}, \ldots, \widetilde M_{k, k - 1}$ are independent Poisson random variables. This parallels the Poisson coupling of Backhausz and Kunszenti-Kov\'acs \cite{Backhausz2017OnTD}, who study a dense preferential-attachment multigraph where vertices arrive non-sequentially, applying the P\'olya urn process on all vertices at once. In this section, we develop a Poisson surrogate for a sequential vertex-arrival model.

\begin{lemma}[Multinomial as conditional Poisson]\label{lem:multi-poisson}
Let $X_1, \dots, X_k$ be independent random variables with $X_j \sim \mathrm{Poisson} \left(\lambda_j\right)$ for $j = 1, \dots, k$, and set $\Lambda := \sum_{j = 1}^k \lambda_j > 0$. Then, for non-negative integers $m, x_1, \ldots, x_k$ that satisfy $\sum_{j = 1}^k x_j = m$,
\[
\P \left[X_1 \ = \ x_1, \dots, X_k \ = \ x_k \,\middle|\, \sum_{j = 1}^k X_j \ = \ m\right]
\ = \
\frac{m!}{x_1! \cdots x_k!} \prod_{j = 1}^k \left(\frac{\lambda_j}{\Lambda}\right)^{x_j}.
\]
More concisely, conditional on $\sum_{j = 1}^k X_j = m$, the random vector $\left(X_1, \dots, X_k\right)$ has the probability distribution described by $\mathrm{Multinomial} \left(m; \frac{\lambda_1}{\Lambda}, \dots, \frac{\lambda_k}{\Lambda}\right)$.
\end{lemma}

\begin{prop}[Perturbation bound for Stieltjes transforms; Lu and Yau {\cite[Lemma 18]{LuEquivalencePrinciple2022}}]\label{prop:perturb-bound}
Let $A, B$ be $n \times n$ Hermitian matrices, and let $z = x + i \eta \in \C$ with $\eta > 0$. Denote their resolvents by $G_A \left(z\right) = \left(A - z I\right)^{-1}$ and $G_B \left(z\right) = \left(B - z I\right)^{-1}$, and their Stieltjes transforms by
\[
s_A \left(z\right)
\ = \
\frac{1}{n} \Tr G_A \left(z\right),
\qquad
s_B \left(z\right)
\ = \
\frac{1}{n} \Tr G_B \left(z\right).
\]
Then
\[
\abs{s_A \left(z\right) - s_B \left(z\right)}
\ \le \
\frac{1}{\left(\Im z\right)^2} \cdot \frac{\norm{A - B}_F}{\sqrt{n}}.
\]
\end{prop}

\begin{prop}\label{prop:poisson-dev}
Let $S \sim \mathrm{Poisson} \left(\mu\right)$. Then
\[
\E\left[\abs{S - \mu}\right]
\ \le \
\sqrt{\mu}.
\]
\end{prop}

\begin{proof}
For any random variable $X$ with mean $\mu$, by the Cauchy-Schwarz inequality we have
\[
\E\left[\abs{X - \mu}\right]
\ = \
\E\left[\abs{X - \mu} \cdot 1\right]
\ \le \
\sqrt{\E\left[\left(X - \mu\right)^2\right]} \sqrt{\E\left[1^2\right]}
\ = \
\sqrt{\Var \left(X\right)}.
\]
Since $\Var \left(S\right) = \mu$ for a Poisson random variable, the claim follows.
\end{proof}

\medskip

\begin{prop}\label{prop:deconditioning}
Let $n \to \infty$ and $m = c n$ with fixed $c \in \left(0, 1 / 2\right)$. Let $M$ denote the symmetric B--A multigraph adjacency matrix, and for each $k > t_0$ let
\[
\mathbf d \left(k - 1\right)
\ := \
\left(d_1 \left(k - 1\right), \dots, d_{k - 1} \left(k - 1\right)\right),
\qquad
p_{k j}
\ := \
\frac{d_j \left(k - 1\right)}{D_{k - 1}},
\qquad j \ < \ k.
\]
Conditionally on $\mathbf d \left(k - 1\right)$, the row-vector $\left(M_{k 1}, \dots, M_{k, k - 1}\right)$ is multinomial with $m$ trials and probabilities $\left(p_{k 1}, \dots, p_{k, k - 1}\right)$. Construct the symmetric matrix $\widetilde M$ using the construction in the proof below, where conditionally independent entries in each newly generated row satisfy
\[
\widetilde M_{k j} \mid \mathbf d \left(k - 1\right)
\ \sim \
\mathrm{Poisson} \left(\lambda_{k j}\right),
\qquad
\lambda_{k j}
\ := \
m p_{k j},
\qquad j \ < \ k,
\]
and set $\widetilde M_{j k} = \widetilde M_{k j}$, $\widetilde M_{k k} = M_{k k} = 0$. For $j \ne k \le t_0$, set $\widetilde M_{k j} = M_{k j} = 1$.

\medskip
Let
\[
Y
\ := \
\frac{M - \E\left[M\right]}{\sqrt{m}},
\qquad
\widetilde Y
\ := \
\frac{\widetilde M - \E\left[\widetilde M\right]}{\sqrt{m}}.
\]
Then $Y$ and $\widetilde Y$ may be coupled so that, for each fixed $z \in \mathbb{H}$,
\[
\abs{s_Y \left(z\right) - s_{\widetilde Y} \left(z\right)}
\ \to \
0
\]
in probability as $n \to \infty$. Moreover,
\[
\abs{\frac{1}{n} \E\left[\Tr \left(Y - z I\right)^{-1}\right] - \frac{1}{n} \E\left[\Tr \left(\widetilde Y - z I\right)^{-1}\right]}
\ = \
o \left(1\right).
\]
Consequently, whenever one of $Y$ and $\widetilde Y$ has a deterministic limiting spectral measure, the other has the same limiting spectral measure.
\end{prop}

\begin{proof}

The matrices $\widetilde M$ and $\overline M$ are constructed sequentially. Start from the same initial clique and let $k > t_0$ be fixed. Suppose both matrices were constructed up to time $k - 1$. For all $j < k$, let
\[
d_j \left(k - 1\right)
\ := \
\sum_{i = 1}^{k - 1} \overline M_{i j},
\]
i.e., the probabilities of attachment at time $k$ are defined by the degrees in $\overline M$. Conditional on these degrees, the $m$ attachments of vertex $k$ are i.i.d.\ draws from $\left\{1, \dots, k - 1\right\}$ with probabilities
\[
p_{k j}
\ := \
\frac{d_j \left(k - 1\right)}{D_{k - 1}},
\qquad
\sum_{j = 1}^{k - 1} p_{k j}
\ = \
1.
\]

Generate independent random variables
\[
\widetilde M_{k j} \mid \mathbf d \left(k - 1\right)
\ \sim \
\mathrm{Poisson} \left(\lambda_{k j}\right),
\qquad
\lambda_{k j}
\ := \
m p_{k j},
\qquad j \ < \ k,
\]
and let
\[
S_k
\ := \
\sum_{j < k} \widetilde M_{k j}.
\]
Since
\[
\sum_{j < k} \lambda_{k j}
\ = \
m,
\]
we have
\[
S_k \mid \mathbf d \left(k - 1\right)
\ \sim \
\mathrm{Poisson} \left(m\right).
\]

If $S_k > m$, then uniformly at random remove $S_k - m$ of the $S_k$ Poisson points. If $S_k < m$, then independently add $m - S_k$ new points, where each chooses the endpoint $j$ with probability $p_{k j}$. Let
\[
\left(\overline M_{k 1}, \dots, \overline M_{k, k - 1}\right)
\]
be the resulting row.

Conditional on $\mathbf d \left(k - 1\right)$ and $S_k$, Lemma \ref{lem:multi-poisson} gives
\[
\left(\widetilde M_{k 1}, \dots, \widetilde M_{k, k - 1}\right) \ \big| \ \left(\mathbf d \left(k - 1\right), S_k\right)
\ \sim \
\mathrm{Multinomial} \left(S_k; p_{k 1}, \dots, p_{k, k - 1}\right).
\]
Therefore, after the deletion or addition procedure,
\[
\left(\overline M_{k 1}, \dots, \overline M_{k, k - 1}\right) \ \big| \ \mathbf d \left(k - 1\right)
\ \sim \
\mathrm{Multinomial} \left(m; p_{k 1}, \dots, p_{k, k - 1}\right).
\]
Proceeding inductively over $k$, we obtain
\[
\overline M
\ \stackrel{d}{=} \
M.
\]
Moreover,
\[
\E\left[\overline M_{k j} \mid \mathbf d \left(k - 1\right)\right]
\ = \
\E\left[\widetilde M_{k j} \mid \mathbf d \left(k - 1\right)\right]
\ = \
m p_{k j},
\qquad
k > t_0,
\qquad
j \ < \ k.
\]
The expectations also agree for the diagonal and the initial clique, and on the remaining entries by symmetry, hence
\[
\E\left[M\right]
\ = \
\E\left[\overline M\right]
\ = \
\E\left[\widetilde M\right].
\]

Let
\[
\overline Y
\ := \
\frac{\overline M - \E\left[\widetilde M\right]}{\sqrt{m}},
\qquad
\widetilde Y
\ := \
\frac{\widetilde M - \E\left[\widetilde M\right]}{\sqrt{m}}.
\]
Then
\[
\overline Y
\ \stackrel{d}{=} \
Y.
\]
Define the ``edit'' difference
\[
\Delta
\ := \
\overline Y - \widetilde Y
\ = \
\frac{\overline M - \widetilde M}{\sqrt{m}}.
\]
Denote
\[
s_{\overline Y} \left(z\right)
\ := \
\frac{1}{n} \Tr \left(\overline Y - z I\right)^{-1},
\qquad
s_{\widetilde Y} \left(z\right)
\ := \
\frac{1}{n} \Tr \left(\widetilde Y - z I\right)^{-1}.
\]
By the perturbation bound for Stieltjes transforms (Proposition \ref{prop:perturb-bound}),
\begin{equation}\label{eq:st-pert}
\abs{s_{\overline Y} \left(z\right) - s_{\widetilde Y} \left(z\right)}
\ \le \
\frac{1}{\left(\Im z\right)^2} \cdot \frac{\norm{\Delta}_F}{\sqrt{n}}.
\end{equation}

We only have to bound $\norm{\Delta}_F$. Let $R_k := \abs{S_k - m}$, and let $C_{k j}$ denote the number of points added to or deleted from coordinate $j$ in row $k$. Then
\[
\sum_{j < k} C_{k j}
\ = \
R_k.
\]
Conditional on $R_k$ and $\mathbf d \left(k - 1\right)$,
\[
\left(C_{k 1}, \dots, C_{k, k - 1}\right)
\ \sim \
\mathrm{Multinomial} \left(R_k; p_{k 1}, \dots, p_{k, k - 1}\right).
\]
Therefore
\[
\E\left[\sum_{j < k} C_{k j}^2 \,\middle|\, R_k, \mathbf d \left(k - 1\right)\right]
\ = \
R_k + R_k \left(R_k - 1\right) \sum_{j < k} p_{k j}^2.
\]

Since each change appears symmetrically in the matrix,
\[
\norm{\Delta}_F^2
\ = \
\frac{2}{m} \sum_{k > t_0} \sum_{j < k} C_{k j}^2.
\]

Now
\[
S_k \mid \mathbf d \left(k - 1\right)
\ \sim \
\mathrm{Poisson} \left(m\right),
\]
and this conditional distribution does not depend on $\mathbf d \left(k - 1\right)$. Hence $R_k$ is independent of $\mathbf d \left(k - 1\right)$, and
\[
\E\left[R_k\right]
\ \le \
\sqrt{m},
\qquad
\E\left[R_k^2\right]
\ = \
m.
\]

Fix $j$ and let $s_0 := \max\left\{j, t_0\right\}$. Since $d_j \left(s_0\right) = m$, we have that $\Var \left(d_j \left(s_0\right)\right) = 0$. For $t > s_0$, conditional on $\mathcal F_{t - 1}$,
\[
\Delta d_j \left(t\right)
\ \sim \
\mathrm{Binomial} \left(m, \frac{d_j \left(t - 1\right)}{D_{t - 1}}\right).
\]
Set
\[
a_t
\ := \
\frac{m}{D_{t - 1}}
\ = \
O \left(\frac{1}{n}\right)
\]
uniformly in $t$. Then
\[
\E\left[d_j \left(t\right) \mid \mathcal F_{t - 1}\right]
\ = \
\left(1 + a_t\right) d_j \left(t - 1\right),
\qquad
\Var \left(d_j \left(t\right) \mid \mathcal F_{t - 1}\right)
\ = \
\Var \left(\Delta d_j \left(t\right) \mid \mathcal F_{t - 1}\right)
\ \le \
a_t d_j \left(t - 1\right).
\]

We also know $\E\left[d_j \left(t\right)\right] = O \left(n\right)$ uniformly in $j$ and $t$. By the law of total variance
\begin{align*}
\Var \left(d_j \left(t\right)\right)
&\ = \
\E\left[\Var \left(d_j \left(t\right) \mid \mathcal F_{t - 1}\right)\right] + \Var \left(\E\left[d_j \left(t\right) \mid \mathcal F_{t - 1}\right]\right) \\
&\ \le \
a_t \E\left[d_j \left(t - 1\right)\right] + \left(1 + a_t\right)^2 \Var \left(d_j \left(t - 1\right)\right) \\
&\ \le \
C + \left(1 + \frac{C}{n}\right) \Var \left(d_j \left(t - 1\right)\right).
\end{align*}
Expanding this recursion for at most $n$ steps gives
\[
\Var \left(d_j \left(t\right)\right)
\ \le \
C \sum_{r = 0}^{n} \left(1 + \frac{C}{n}\right)^r
\ = \
O \left(n\right).
\]
In particular,
\[
\Var \left(d_j \left(k - 1\right)\right)
\ = \
O \left(n\right).
\]

Using this result, together with $\E\left[d_j \left(k - 1\right)\right] = O \left(n\right)$ and $D_{k - 1} = \Theta \left(n^2\right)$, we obtain
\[
\E\left[\sum_{j < k} p_{k j}^2\right]
\ = \
\frac{1}{D_{k - 1}^2} \sum_{j < k} \E\left[d_j \left(k - 1\right)^2\right]
\ = \
O \left(\frac{1}{n}\right).
\]
Therefore
\begin{align*}
\E\left[\sum_{j < k} C_{k j}^2\right]
&\ = \
\E\left[R_k\right] + \E\left[R_k \left(R_k - 1\right) \sum_{j < k} p_{k j}^2\right] \\
&\ \le \
\sqrt{m} + \E\left[R_k^2\right] \E\left[\sum_{j < k} p_{k j}^2\right] \\
&\ = \
O \left(\sqrt{n}\right).
\end{align*}
Hence
\[
\E\left[\norm{\Delta}_F^2\right]
\ = \
\frac{2}{m} \sum_{k > t_0} O \left(\sqrt{n}\right)
\ = \
O \left(\sqrt{n}\right),
\]
and consequently
\[
\E\left[\norm{\Delta}_F\right]
\ \le \
\sqrt{\E\left[\norm{\Delta}_F^2\right]}
\ = \
O \left(n^{1 / 4}\right).
\]

For fixed $z \in \mathbb{H}$, taking expectations in \eqref{eq:st-pert} gives
\[
\E\left[\abs{s_{\overline Y} \left(z\right) - s_{\widetilde Y} \left(z\right)}\right]
\ \le \
\frac{1}{\left(\Im z\right)^2} \frac{O \left(n^{1 / 4}\right)}{\sqrt{n}}
\ = \
O \left(n^{-1 / 4}\right)
\ \to \
0.
\]
Hence, by Markov's inequality,
\[
\abs{s_{\overline Y} \left(z\right) - s_{\widetilde Y} \left(z\right)}
\ \to \
0
\]
in probability.

Since $\overline Y \stackrel{d}{=} Y$, this coupling shows that $Y$ and $\widetilde Y$ may be coupled so that
\[
\abs{s_Y \left(z\right) - s_{\widetilde Y} \left(z\right)}
\ \to \
0
\]
in probability. Moreover,
\begin{align*}
\abs{\frac{1}{n} \E\left[\Tr \left(Y - z I\right)^{-1}\right] - \frac{1}{n} \E\left[\Tr \left(\widetilde Y - z I\right)^{-1}\right]}
&\ = \
\abs{\E\left[s_{\overline Y} \left(z\right)\right] - \E\left[s_{\widetilde Y} \left(z\right)\right]} \\
&\ \le \
\E\left[\abs{s_{\overline Y} \left(z\right) - s_{\widetilde Y} \left(z\right)}\right] \\
&\ \to \
0.
\end{align*}
Thus, whenever one of $Y$ and $\widetilde Y$ has a deterministic limiting spectral measure, the other has the same limiting spectral measure.
\end{proof}

\begin{prop}[Independent Poisson comparison]\label{prop:independent-poisson-comparison}
For $k > t_0$ and $j < k$, define
\[
\bar \lambda_{k j}
\ := \
\E\left[\widetilde M_{k j}\right]
\ = \
\E\left[M_{k j}\right].
\]
Let $P_{k j}$ be independent random variables with
\[
P_{k j}
\ \sim \
\mathrm{Poisson} \left(\bar \lambda_{k j}\right),
\qquad
k \ > \ t_0, \ j \ < \ k.
\]
Define the symmetric random matrix $H$ as
\[
H_{k j}
\ := \
\frac{P_{k j} - \bar \lambda_{k j}}{\sqrt{m}},
\qquad
H_{j k}
\ := \
H_{k j},
\qquad
k \ > \ t_0, \ j \ < \ k,
\]
set $H_{k k} = 0$ for all $k$, and set $H_{k j} = 0$ for off-diagonal entries in the initial clique, i.e., when $k \ne j$ and $k, j \le t_0$. Then $\widetilde Y$ and $H$ can be coupled such that, for every fixed $z \in \mathbb H$,
\[
\abs{s_{\widetilde Y} \left(z\right) - s_H \left(z\right)}
\ \to \
0
\]
in probability. Consequently, whenever one of $\widetilde Y$ and $H$ has a deterministic limiting spectral measure, the other has the same limiting spectral measure.
\end{prop}

\begin{proof}
For $k > t_0$ and $j < k$, write
\[
\lambda_{k j}
\ := \
m p_{k j}
\ = \
m \frac{d_j \left(k - 1\right)}{D_{k - 1}},
\qquad
\bar \lambda_{k j}
\ := \
\E\left[\lambda_{k j}\right].
\]
It is enough to compare $H$ with a random matrix $\widehat Y$ whose entries have the same joint distribution as $\widetilde Y$, because the empirical spectral distribution depends only on the law of the matrix. Let
\[
\widehat Y
\ := \
\frac{\widehat M - \E\left[\widetilde M\right]}{\sqrt{m}}.
\]
We construct $\widehat M$ and $P$ jointly and sequentially so that $\widehat M$ has the same law as $\widetilde M$, while $P$ has independent Poisson entries.

Note that if $A$ and $B$ are independent Poisson random variables with means $\alpha$ and $\beta$, then $A + B$ is Poisson with mean $\alpha + \beta$.

At step $k$, conditional on the history up to time $k - 1$, the quantities $\lambda_{k j}$ are determined. For $j < k$, define
\[
\alpha_{k j}
\ := \
\min\left\{\lambda_{k j}, \bar \lambda_{k j}\right\},
\qquad
\beta_{k j}
\ := \
\abs{\lambda_{k j} - \bar \lambda_{k j}}.
\]
Conditional on the history up to time $k - 1$, let $A_{k j}$ and $B_{k j}$ be independent Poisson random variables with means $\alpha_{k j}$ and $\beta_{k j}$, respectively, independently over $j$.

If $\lambda_{k j} \ge \bar \lambda_{k j}$, define
\[
P_{k j}
\ := \
A_{k j},
\qquad
\widehat M_{k j}
\ := \
A_{k j} + B_{k j}.
\]
If $\lambda_{k j} < \bar \lambda_{k j}$, define
\[
\widehat M_{k j}
\ := \
A_{k j},
\qquad
P_{k j}
\ := \
A_{k j} + B_{k j}.
\]
In either case, conditional on the history up to time $k - 1$,
\[
\widehat M_{k j}
\ \sim \
\mathrm{Poisson} \left(\lambda_{k j}\right),
\qquad
P_{k j}
\ \sim \
\mathrm{Poisson} \left(\bar \lambda_{k j}\right).
\]
Since $\bar \lambda_{k j}$ is deterministic, the conditional law of $P_{k j}$ does not depend on the past. Therefore, proceeding sequentially over $k$, the variables $P_{k j}$ are independent Poisson random variables with means $\bar \lambda_{k j}$.

Using the same sequential construction as in the proof of Proposition \ref{prop:deconditioning} to generate the degree history, $\widehat M$ has the same law as $\widetilde M$. Hence
\[
\widehat Y
\ \stackrel{d}{=} \
\widetilde Y.
\]

Now
\[
\widehat M_{k j} - P_{k j}
\ = \
\pm B_{k j},
\]
so
\[
\left(\widehat M_{k j} - P_{k j}\right)^2
\ = \
B_{k j}^2.
\]
Conditional on the history up to time $k - 1$,
\[
B_{k j}
\ \sim \
\mathrm{Poisson} \left(\beta_{k j}\right),
\]
and therefore
\[
\E\left[\left(\widehat M_{k j} - P_{k j}\right)^2 \,\middle|\, \mathbf d \left(k - 1\right)\right]
\ = \
\beta_{k j} + \beta_{k j}^2.
\]
Thus
\[
\E\left[\left(\widehat M_{k j} - P_{k j}\right)^2\right]
\ = \
\E\left[\abs{\lambda_{k j} - \bar \lambda_{k j}}\right] + \E\left[\abs{\lambda_{k j} - \bar \lambda_{k j}}^2\right].
\]

From the proof of Proposition \ref{prop:deconditioning},
\[
\Var \left(d_j \left(k - 1\right)\right)
\ = \
O \left(n\right).
\]
Since $D_{k - 1}$ is deterministic and $D_{k - 1} = \Theta \left(n^2\right)$,
\[
\Var \left(\lambda_{k j}\right)
\ = \
\Var \left(m \frac{d_j \left(k - 1\right)}{D_{k - 1}}\right)
\ = \
\frac{\left(c n\right)^2}{D_{k - 1}^2} \Var \left(d_j \left(k - 1\right)\right)
\ = \
O \left(\frac{1}{n}\right).
\]
Therefore
\[
\E\left[\abs{\lambda_{k j} - \bar \lambda_{k j}}^2\right]
\ = \
\Var \left(\lambda_{k j}\right)
\ = \
O \left(\frac{1}{n}\right),
\]
and by Cauchy-Schwarz,
\[
\E\left[\abs{\lambda_{k j} - \bar \lambda_{k j}}\right]
\ \le \
\sqrt{\E\left[\abs{\lambda_{k j} - \bar \lambda_{k j}}^2\right]}
\ = \
O \left(\frac{1}{\sqrt{n}}\right).
\]
Hence
\[
\E\left[\left(\widehat M_{k j} - P_{k j}\right)^2\right]
\ = \
O \left(\frac{1}{\sqrt{n}}\right).
\]

Since
\[
\widehat Y - H
\ = \
\frac{\widehat M - P}{\sqrt{m}}
\]
on the non-deterministic off-diagonal entries, and since deterministic entries contribute zero, we have
\[
\E\left[\norm{\widehat Y - H}_F^2\right]
\ \le \
\frac{2}{m} \sum_{k > t_0} \sum_{j < k} \E\left[\left(\widehat M_{k j} - P_{k j}\right)^2\right]
\ \le \
\frac{2}{c n} \cdot n^2 \cdot O \left(\frac{1}{\sqrt{n}}\right)
\ = \
O \left(\sqrt{n}\right).
\]
Therefore
\[
\E\left[\norm{\widehat Y - H}_F\right]
\ \le \
\sqrt{\E\left[\norm{\widehat Y - H}_F^2\right]}
\ = \
O \left(n^{1 / 4}\right).
\]
By Proposition \ref{prop:perturb-bound},
\[
\abs{s_{\widehat Y} \left(z\right) - s_H \left(z\right)}
\ \le \
\frac{1}{\left(\Im z\right)^2} \frac{\norm{\widehat Y - H}_F}{\sqrt{n}}.
\]
Taking expectations gives
\[
\E\left[\abs{s_{\widehat Y} \left(z\right) - s_H \left(z\right)}\right]
\ \le \
\frac{1}{\left(\Im z\right)^2} \frac{O \left(n^{1 / 4}\right)}{\sqrt{n}}
\ = \
O \left(n^{-1 / 4}\right)
\ \to \
0.
\]
Hence
\[
\abs{s_{\widehat Y} \left(z\right) - s_H \left(z\right)}
\ \to \
0
\]
in probability. Since $\widehat Y$ has the same law as $\widetilde Y$, this coupling shows that
\[
\abs{s_{\widetilde Y} \left(z\right) - s_H \left(z\right)}
\ \to \
0
\]
in probability. Therefore, whenever one of $\widetilde Y$ and $H$ has a deterministic limiting spectral measure, the other has the same limiting spectral measure.
\end{proof}

\subsection{Variance for the centered matrix}\label{subsec:variance}

\begin{prop}[Variance kernel for the independent centered matrix]\label{prop:variance-kernel}
Let $H$ be the independent Poisson matrix from Proposition~\ref{prop:independent-poisson-comparison}, with $m = c n$ and fixed $c \in \left(0, 1 / 2\right)$. For $k > t_0 := m + 1$ and $j < k$, set
\[
p_{k j}
\ := \
\frac{d_j \left(k - 1\right)}{D_{k - 1}},
\qquad
D_{k - 1}
\ := \
\sum_{v = 1}^{k - 1} d_v \left(k - 1\right).
\]
Then
\[
\Var \left(H_{k j}\right)
\ = \
\E\left[p_{k j}\right],
\qquad
n \Var \left(H_{k j}\right)
\ = \
n \E\left[p_{k j}\right].
\]
Consequently, for index sequences $j = j_n$ and $k = k_n$ that satisfy $j / n \to x$ and $k / n \to y$, with $x \ne y$ and $x, y \ne c$, we have
\begin{equation}\label{eq:sigma}
\sigma^2 \left(x, y\right)
\ := \
\lim_{n \to \infty} n \Var \left(H_{k j}\right)
\ = \
\begin{cases}
0, & x, y \in A,\\
\dfrac{1}{c} F \left(x, y\right), & \text{otherwise}.
\end{cases}
\end{equation}

The diagonal values do not affect the associated integral operator, so we obtain the following, where $A = \left[0, c\right]$, $B = \left(c, 1\right]$, and $u \left(t\right) = \left(2 t - c\right)^{-1 / 2}$:
\begin{equation}\label{eq:sigma-piecewise}
\sigma^2 \left(x, y\right)
\ = \
\begin{cases}
0, & x, y \in A,\\
\dfrac{1}{\sqrt{c}} u \left(y\right), & x \in A, \ y \in B,\\
\dfrac{1}{\sqrt{c}} u \left(x\right), & x \in B, \ y \in A,\\
u \left(x\right) u \left(y\right), & x, y \in B,
\end{cases}
\qquad
u \left(t\right)
\ = \
\left(2 t - c\right)^{-1 / 2}.
\end{equation}
\end{prop}

\begin{proof}
Let $H$ be the independent Poisson matrix from Proposition~\ref{prop:independent-poisson-comparison}. For $k > t_0$ and $j < k$, write
\[
\bar \lambda_{k j}
\ := \
\E\left[M_{k j}\right].
\]
By definition,
\[
H_{k j}
\ = \
\frac{P_{k j} - \bar \lambda_{k j}}{\sqrt{m}},
\qquad
P_{k j}
\ \sim \
\mathrm{Poisson} \left(\bar \lambda_{k j}\right).
\]
Therefore
\[
\Var \left(H_{k j}\right)
\ = \
\frac{1}{m} \Var \left(P_{k j}\right)
\ = \
\frac{\bar \lambda_{k j}}{m}
\ = \
\frac{\E\left[M_{k j}\right]}{m}.
\]
Since $m = c n$, this gives
\[
n \Var \left(H_{k j}\right)
\ = \
\frac{n}{m} \E\left[M_{k j}\right]
\ = \
\frac{1}{c} \E\left[M_{k j}\right].
\]

If $k, j \le t_0$, then $H_{k j}$ is deterministic and equal to $0$, so
\[
\Var \left(H_{k j}\right)
\ = \
0.
\]
Then the variance kernel is zero on $A \times A$.

Now assume that at least one vertex is outside of the initial clique. By symmetry, it is enough to consider $k > t_0$ and $j < k$. Assuming the restrictions in the Proposition statement, by Theorem \ref{thm:mean-kernel-limit}, with $j / n \to x$ and $k / n \to y$,
\[
\E\left[M_{k j}\right]
\ = \
F \left(k / n, j / n\right) + O \left(n^{-1}\right)
\ \to \
F \left(x, y\right).
\]
Therefore
\[
n \Var \left(H_{k j}\right)
\ = \
\frac{1}{c} \E\left[M_{k j}\right]
\ \to \
\frac{1}{c} F \left(x, y\right).
\]
Hence
\[
\lim_{n \to \infty} n \Var \left(H_{k j}\right)
\ = \
\frac{1}{c} F \left(x, y\right)
\]
outside $A \times A$.

Substituting \eqref{eq:FWR} gives \eqref{eq:sigma-piecewise} for $x \ne y$ and $x, y \ne c$. We choose the remaining values as in \eqref{eq:sigma-piecewise}. The diagonal and the lines $x = c$ and $y = c$ have Lebesgue measure zero, so they do not affect the associated the integral
\[
\int_0^1 \sigma^2 \left(x, y\right) m \left(y; z\right) \, d y
\]
in \eqref{eq:QVE-cont} for almost every $x \in \left[0, 1\right]$.
\end{proof}

\begin{prop}\label{prop:sigma-rank-two}
Define the integral operator $S : L^2 \left(\left[0, 1\right]\right) \to L^2 \left(\left[0, 1\right]\right)$ as
\[
\left(S f\right) \left(x\right)
\ := \
\int_0^1 \sigma^2 \left(x, y\right) f \left(y\right) \, d y,
\]
with $\sigma^2$ is given by \eqref{eq:sigma-piecewise}. Then $S$ is rank-two.
\end{prop}

\begin{proof}
Define
\[
\phi_1 \left(x\right)
\ := \
\begin{cases}
1, & x \in A,\\
0, & x \in B,
\end{cases}
\qquad
\phi_2 \left(x\right)
\ := \
\begin{cases}
0, & x \in A,\\
u \left(x\right), & x \in B.
\end{cases}
\]
Then $\sigma^2$ from \eqref{eq:sigma-piecewise} can be expressed using only $\phi_1$ and $\phi_2$:
\[
\sigma^2 \left(x, y\right)
\ = \
\frac{1}{\sqrt{c}} \phi_1 \left(x\right) \phi_2 \left(y\right) + \frac{1}{\sqrt{c}} \phi_2 \left(x\right) \phi_1 \left(y\right) + \phi_2 \left(x\right) \phi_2 \left(y\right).
\]

This shows that $S f \in \operatorname{span}\left\{\phi_1, \phi_2\right\}$ for all $f \in L^2 \left(\left[0, 1\right]\right)$. Hence $\operatorname{rank} \left(S\right) \le 2$.

Define $a := \int_c^1 u \left(y\right)^2 \, d y > 0$. Then
\[
S \phi_1
\ = \
\sqrt{c} \phi_2,
\qquad
S \phi_2
\ = \
a \left(\frac{\phi_1}{\sqrt{c}} + \phi_2\right).
\]
They are linearly independent, since $\phi_1$ and $\phi_2$ have disjoint, nonempty supports, so $\operatorname{rank} \left(S\right) = 2$.
\end{proof}

\subsection{Quadratic vector equations (QVE)}

Let $U, V : \left[0, 1\right]^2 \to \R$ be bounded measurable kernels. Define their cut distance as
\[
\delta_{\square} \left(U, V\right)
\ := \
\inf_{\pi} \sup_{E, F} \abs{\int_{E \times F} \left(U \left(\pi \left(x\right), \pi \left(y\right)\right) - V \left(x, y\right)\right) \, d x \, d y},
\]
where $\pi$ is taken over all bijections of $\left[0, 1\right]$ that preserve the Lebesgue measure, and $E, F$ are taken over all measurable subsets of $\left[0, 1\right]$; see \cite[Definition 2.3]{Zhu2018AGA}. Setting $\pi$ to be the identity map and taking absolute values inside the integral results in the bound
\begin{equation}\label{eq:cut-dist-bound}
\delta_{\square} \left(U, V\right)
\ \le \
\int_{\left[0, 1\right]^2} \abs{U \left(x, y\right) - V \left(x, y\right)} \, d x \, d y
\ = \
\norm{U - V}_{L^1 \left(\left[0, 1\right]^2\right)}.
\end{equation}

\begin{prop}[Zhu {\cite[Theorem 3.4]{Zhu2018AGA}}]
\label{prop:Zhu}
Let $X^{\left(n\right)} = \left(X^{\left(n\right)}_{k j}\right)_{k, j = 1}^n$ be real symmetric random matrices. Let
\[
I_1^{\left(n\right)}
\ = \
\left[0, 1 / n\right],
\qquad
I_k^{\left(n\right)}
\ = \
\left(\left(k - 1\right) / n, k / n\right],
\qquad
2 \ \le \ k \ \le \ n,
\]
and
\[
W_n \left(x, y\right)
\ := \
\E\left[\abs{X^{\left(n\right)}_{k j}}^2\right],
\qquad
x \in I_k^{\left(n\right)},
\qquad
y \in I_j^{\left(n\right)}.
\]
Assume that the following are true.
\begin{itemize}
\item[(A)] The entries are centered at zero and independent, up to diagonal symmetry.

\item[(B)] There exists a fixed $C < \infty$ independent of $n, k, j$ such that
\[
\E\left[\abs{X^{\left(n\right)}_{k j}}^2\right]
\ \le \
C.
\]

\item[(C)] For all $\varepsilon > 0$,
\[
\lim_{n \to \infty} \frac{1}{n^2} \sum_{k, j = 1}^n \E\left[\abs{X^{\left(n\right)}_{k j}}^2 \mathbf 1_{\left\{\abs{X^{\left(n\right)}_{k j}} \ \ge \ \varepsilon \sqrt{n}\right\}}\right]
\ = \
0.
\]

\item[(D)] There exists a bounded, symmetric, nonnegative measurable kernel $W : \left[0, 1\right]^2 \to \R$ such that
\[
\lim_{n \to \infty} \delta_{\square} \left(W_n, W\right)
\ = \
0.
\]
\end{itemize}

Then the empirical eigenvalue distribution of $X^{\left(n\right)} / \sqrt{n}$ converges weakly almost surely\footnote{A sequence $\left(\mu_n\right)$ of random probability measures $\mu_n$ over the real numbers converges weakly almost surely to some deterministic probability measure $\mu$ if, for all bounded continuous functions $f : \R \to \R$,
\[
\P \left(\int_{\R} f \, d \mu_n \ \to \ \int_{\R} f \, d \mu\right)
\ = \
1.
\]} to a probability measure $\mu$. The Stieltjes transform of $\mu$ is given by
\[
G_\mu \left(z\right)
\ := \
\int_{\R} \frac{1}{\lambda - z} \, d \mu \left(\lambda\right)
\ = \
\int_0^1 m_W \left(x; z\right) \, d x,
\qquad
z \in \mathbb H.
\]
For each $z \in \mathbb H$, $m_W \left(\cdot ; z\right) : \left[0, 1\right] \to \mathbb H$ is the unique bounded measurable solution of the following quadratic vector equation (QVE):
\[
- \frac{1}{m_W \left(x; z\right)}
\ = \
z + \int_0^1 W \left(x, y\right) m_W \left(y; z\right) \, d y,
\qquad x \in \left[0, 1\right].
\]
\end{prop}

Zhu \cite{Zhu2018AGA} uses the Stieltjes-transform convention $\int_{\R} \left(z - \lambda\right)^{-1} \, d \mu \left(\lambda\right)$. Our expression follows by changing the signs of that transform and of Zhu's solution $a \left(z, x\right)$.

\begin{prop}[Verification of Zhu's assumptions]
\label{prop:check-sigma-conditions}
Let $H$ be the matrix from Proposition \ref{prop:independent-poisson-comparison}, with $m = c n$ and fixed rational $c \in \left(0, 1 / 2\right)$. Then the matrix
\[
X^{\left(n\right)}
\ := \
\sqrt{n} H
\]
satisfies conditions \textnormal{(A)--(D)} of Proposition \ref{prop:Zhu}, where $W = \sigma^2$ from \eqref{eq:sigma-piecewise}.
\end{prop}

\begin{proof}
For $k > t_0$ and $j < k$, the definition of $H$ gives
\[
X^{\left(n\right)}_{k j}
\ = \
\frac{P_{k j} - \bar \lambda_{k j}}{\sqrt{c}},
\qquad
P_{k j}
\ \sim \
\mathrm{Poisson} \left(\bar \lambda_{k j}\right),
\qquad
\bar \lambda_{k j}
\ = \
\E\left[M_{k j}\right].
\]
So the entries are independent (up to symmetry) or are zero. Since $\E\left[P_{k j}\right] = \bar \lambda_{k j}$, condition (A) is true.

By Theorem \ref{thm:mean-kernel-limit}, there exists a constant $C_0 < \infty$ that depends only on $c$ such that, uniformly in $n, k, j$, we have $\bar \lambda_{k j} \le C_0$. So
\[
\E\left[\abs{X^{\left(n\right)}_{k j}}^2\right]
\ = \
\frac{\bar \lambda_{k j}}{c}
\ \le \
\frac{C_0}{c}
\]
for the entries that are not deterministic. The second moment of the constant entries is zero, so the bound holds, too, which proves (B).

To show (C), let $P \sim \mathrm{Poisson} \left(\lambda\right)$. Taking the fourth derivative of
\[
\E\left[e^{t \left(P - \lambda\right)}\right]
\ = \
\exp \left(\lambda \left(e^t - 1 - t\right)\right)
\]
at $t = 0$ gives
\[
\E\left[\left(P - \lambda\right)^4\right]
\ = \
\lambda + 3 \lambda^2.
\]
So, for the non-deterministic entries,
\[
\E\left[\abs{X^{\left(n\right)}_{k j}}^4\right]
\ = \
\frac{\bar \lambda_{k j} + 3 \bar \lambda_{k j}^2}{c^2}
\ \le \
C_4,
\qquad
C_4
\ := \
\frac{C_0 + 3 C_0^2}{c^2}.
\]
This is also true for the constant zero entries. For all $\varepsilon > 0$, we therefore have
\[
\frac{1}{n^2} \sum_{k, j = 1}^n \E\left[\abs{X^{\left(n\right)}_{k j}}^2 \mathbf 1_{\left\{\abs{X^{\left(n\right)}_{k j}} \ \ge \ \varepsilon \sqrt{n}\right\}}\right]
\ \le \
\frac{1}{\varepsilon^2 n^3} \sum_{k, j = 1}^n \E\left[\abs{X^{\left(n\right)}_{k j}}^4\right]
\ \le \
\frac{C_4}{\varepsilon^2 n}
\ \to \
0,
\]
as $n \to \infty$, which proves (C).

Finally, the variance kernel $W_n$ is
\[
\sigma_n^2 \left(x, y\right)
\ := \
n \E\left[\abs{H_{k j}}^2\right],
\qquad
x \in I_k^{\left(n\right)},
\qquad
y \in I_j^{\left(n\right)}.
\]
By Proposition \ref{prop:variance-kernel},
\[
\sigma_n^2 \left(x, y\right)
\ \to \
\sigma^2 \left(x, y\right)
\]
for almost all $\left(x, y\right) \in \left[0, 1\right]^2$. Using (B) we get $0 \le \sigma_n^2 \le C_0 / c$, while \eqref{eq:sigma-piecewise} defines a bounded, symmetric, nonnegative kernel $\sigma^2$.

By dominated convergence,
\begin{equation}\label{eq:sigma-L1-convergence}
\norm{\sigma_n^2 - \sigma^2}_{L^1 \left(\left[0, 1\right]^2\right)}
\ \to \
0.
\end{equation}

By \eqref{eq:cut-dist-bound},
\[
\delta_{\square} \left(\sigma_n^2, \sigma^2\right)
\ \le \
\norm{\sigma_n^2 - \sigma^2}_{L^1 \left(\left[0, 1\right]^2\right)}
\ \to \
0,
\]
proving (D).
\end{proof}

\begin{prop}[Limiting quadratic vector equations for the B--A multigraph]
\label{prop:QVE-BA}
Let $Y = \left(M - \E\left[M\right]\right) / \sqrt{m}$, with $m = c n$ and rational $c \in \left(0, 1 / 2\right)$. Then the empirical spectral distribution of $Y$ converges weakly in probability to a deteministic probability measure $\mu_c$, given by its Stieltjes transform:
\[
G_c \left(z\right)
\ := \
\int_{\R} \frac{1}{\lambda - z} \, d \mu_c \left(\lambda\right)
\ = \
\int_0^1 m \left(x; z\right) \, d x.
\]
For each $z \in \mathbb H$, $m \left(\cdot; z\right) : \left[0, 1\right] \to \mathbb H$ is the unique bounded measurable solution of
\begin{equation}\label{eq:QVE-cont}
- \frac{1}{m \left(x; z\right)}
\ = \
z + \int_0^1 \sigma^2 \left(x, y\right) m \left(y; z\right) \, d y,
\qquad x \in \left[0, 1\right].
\end{equation}
\end{prop}

\begin{proof}
From Propositions \ref{prop:check-sigma-conditions} and \ref{prop:Zhu}, the empirical eigenvalue distribution of $X^{\left(n\right)} / \sqrt{n} = H$ converges weakly almost surely, and hence in probaility, to the deterministic probability measure $\mu_c$ defined above. Proposition \ref{prop:independent-poisson-comparison} transfers this limiting spectral distribution to $\widetilde Y$. Proposition \ref{prop:deconditioning} transfers this limiting spectral distribution to $Y$.
\end{proof}

\subsection{Reduction of the QVE}

Recall that the variance kernel $\sigma^2$ from \eqref{eq:sigma-piecewise} is of the form
\[
\sigma^2 \left(x, y\right)
\ = \
\begin{cases}
0, & x, y \in A \ = \ \left[0, c\right],\\
\dfrac{1}{\sqrt{c}} u \left(y\right), & x \in A, \ y \in B,\\
\dfrac{1}{\sqrt{c}} u \left(x\right), & x \in B, \ y \in A,\\
u \left(x\right) u \left(y\right), & x, y \in B,
\end{cases}
\qquad
u \left(t\right)
\ = \
\left(2 t - c\right)^{-1 / 2}.
\]
By Proposition \ref{prop:QVE-BA}, the QVE associated with $\sigma^2$ is
\begin{equation}\label{eq:QVE-continuum}
- \frac{1}{m \left(x; z\right)}
\ = \
z + \int_0^1 \sigma^2 \left(x, y\right) m \left(y; z\right) \, d y,
\qquad x \in \left[0, 1\right], \qquad z \in \mathbb{H}.
\end{equation}
Since the integral associated with $\sigma^2$ is rank-two by Proposition \ref{prop:sigma-rank-two}, this integral equation reduces to a system of two scalar quantities.

Define
\[
M_A \left(z\right)
\ := \
\int_0^c m \left(y; z\right) \, d y,
\qquad
L \left(z\right)
\ := \
\int_c^1 u \left(y\right) m \left(y; z\right) \, d y.
\]
These two scalar quantities capture all dependence of \eqref{eq:QVE-continuum} on $m \left(\cdot; z\right)$, because $\sigma^2$ only couples $m$ through the functions $\phi_1$ and $\phi_2$ defined in Proposition \ref{prop:sigma-rank-two}. We can now prove Theorem \ref{thm:bulk-spectrum}.

\begin{proof}
For $x \in A$, $\sigma^2 \left(x, y\right) = 0$ for $y \in A$ and $\sigma^2 \left(x, y\right) = \left(1 / \sqrt{c}\right) u \left(y\right)$ for $y \in B$. Hence
\[
\int_0^1 \sigma^2 \left(x, y\right) m \left(y; z\right) \, d y
\ = \
\frac{1}{\sqrt{c}} \int_c^1 u \left(y\right) m \left(y; z\right) \, d y
\ = \
\frac{L \left(z\right)}{\sqrt{c}}.
\]
The right-hand side is independent of $x \in A$, so the solution $m \left(x; z\right)$ is constant on $A$, which we denote by
\[
m_A \left(z\right)
\ := \
m \left(x; z\right),
\qquad
x \in A.
\]
Substituting into \eqref{eq:QVE-continuum} gives
\[
- \frac{1}{m_A \left(z\right)}
\ = \
z + \frac{L \left(z\right)}{\sqrt{c}},
\qquad
m_A \left(z\right)
\ = \
- \frac{1}{z + L \left(z\right) / \sqrt{c}}.
\]
Integrating over $A = \left[0, c\right]$ gives $M_A \left(z\right) = c m_A \left(z\right)$.

For $x \in B$,
\[
\sigma^2 \left(x, y\right)
\ = \
\begin{cases}
\dfrac{1}{\sqrt{c}} u \left(x\right), & y \in A,\\
u \left(x\right) u \left(y\right), & y \in B,
\end{cases}
\]
so
\[
\int_0^1 \sigma^2 \left(x, y\right) m \left(y; z\right) \, d y
\ = \
\frac{1}{\sqrt{c}} u \left(x\right) M_A \left(z\right) + u \left(x\right) L \left(z\right)
\ = \
u \left(x\right) \left(L \left(z\right) + \frac{M_A \left(z\right)}{\sqrt{c}}\right).
\]
The QVE becomes
\[
- \frac{1}{m \left(x; z\right)}
\ = \
z + u \left(x\right) \left(L \left(z\right) + \frac{M_A \left(z\right)}{\sqrt{c}}\right),
\]
hence
\[
m \left(x; z\right)
\ = \
- \frac{1}{z + u \left(x\right) K \left(z\right)},
\qquad
K \left(z\right)
\ := \
L \left(z\right) + \frac{M_A \left(z\right)}{\sqrt{c}}.
\]

Using $M_A \left(z\right) = c m_A \left(z\right) = - c / \left(z + L \left(z\right) / \sqrt{c}\right)$, we get
\[
K \left(z\right)
\ = \
L \left(z\right) + \frac{c}{\sqrt{c}} \left(- \frac{1}{z + L \left(z\right) / \sqrt{c}}\right)
\ = \
L \left(z\right) - \frac{\sqrt{c}}{z + L \left(z\right) / \sqrt{c}}.
\]

By definition of $L$,
\[
L \left(z\right)
\ = \
\int_c^1 u \left(y\right) m \left(y; z\right) \, d y
\ = \
- \int_c^1 \frac{u \left(y\right)}{z + u \left(y\right) K \left(z\right)} \, d y.
\]

We also have to prove uniqueness. Pick $z \in \mathbb H$. Let $\left(K, L\right) \in \mathbb H^2$ be a solution of \eqref{eq:QVE-system}. Define
\[
m \left(x; z\right)
\ := \
\begin{cases}
- \dfrac{1}{z + L / \sqrt{c}}, & x \in A,\\[6pt]
- \dfrac{1}{z + u \left(x\right) K}, & x \in B.
\end{cases}
\]
We have $u \left(x\right) > 0$, so the imaginary parts of both denominators are at least $\Im z$, hence $m \left(\cdot; z\right)$ is bounded and is in $\mathbb H$. The system \eqref{eq:QVE-system} gives
\[
\int_c^1 u \left(y\right) m \left(y; z\right) \, d y
\ = \
L,
\qquad
L + \sqrt{c} m_A
\ = \
K,
\]
with $m_A = - 1 / \left(z + L / \sqrt{c}\right)$. Substituting these into \eqref{eq:sigma-piecewise} shows that $m$ satisfies \eqref{eq:QVE-cont}. By the uniqueness in Proposition \ref{prop:Zhu}, $m$ is unique, hence so are $L$ and $K$.

This proves the reduced system \eqref{eq:QVE-system}. The Stieltjes transform of the limiting eigenvalue distribution is
\[
G_c \left(z\right)
\ = \
\int_0^1 m \left(x; z\right) \, d x
\ = \
c m_A \left(z\right) + \int_c^1 - \frac{1}{z + u \left(x\right) K \left(z\right)} \, d x.
\]
\end{proof}

\section{The outlier eigenvalue}

In this section we describe the leading eigenvalue of the adjacency matrix.

\begin{prop}[Weyl's perturbation theorem; {\cite[III.2.6]{BhatiaMatrixAnalysis}}]\label{prop:weyl}
Let $A$ and $B$ be Hermitian $n \times n$ matrices, and let $\lambda_1 \left(\cdot\right) \ge \cdots \ge \lambda_n \left(\cdot\right)$ denote their eigenvalues. Then
\[
\max_{1 \le j \le n} \abs{\lambda_j \left(A\right) - \lambda_j \left(B\right)}
\ \le \
\norm{A - B}.
\]
\end{prop}

\subsection{Top eigenvalue of $\E\left[M\right]$}

Recall from Lemma \ref{lem:EM-rank1} that, for $x \ne y$, the limiting mean kernel satisfies $F \left(x, y\right) = \phi \left(x\right) \phi \left(y\right)$, where
\[
\phi \left(x\right)
\ := \
\begin{cases}
1, & x \in \left[0, c\right],\\
\sqrt{c} \left(2 x - c\right)^{-1 / 2}, & x \in \left(c, 1\right].
\end{cases}
\]
Define the vector $v^{\left(n\right)} \in \R^n$ by
\[
v^{\left(n\right)}_i
\ := \
\phi \left(\frac{i}{n}\right),
\qquad
1 \ \le \ i \ \le \ n.
\]

\begin{lemma}\label{lem:top-eig-EM}
We have
\[
\norm{\E\left[M\right] - v^{\left(n\right)} \left(v^{\left(n\right)}\right)^\top}
\ = \
O \left(1\right).
\]
Consequently, if
\[
\lambda_1 \left(\E\left[M\right]\right)
\ \ge \
\cdots
\ \ge \
\lambda_n \left(\E\left[M\right]\right)
\]
are the eigenvalues of $\E\left[M\right]$, then
\[
\lambda_1 \left(\E\left[M\right]\right)
\ = \
\alpha \left(c\right) n + O \left(1\right),
\]
and
\[
\max_{2 \le j \le n} \abs{\lambda_j \left(\E\left[M\right]\right)}
\ = \
O \left(1\right),
\]
where
\begin{equation}\label{eq:alpha-def}
\alpha \left(c\right)
\ := \
\int_0^1 \phi \left(x\right)^2 \, d x
\ = \
c + \frac{c}{2} \log \frac{2 - c}{c}.
\end{equation}
\end{lemma}

\begin{proof}
By Theorem \ref{thm:mean-kernel-limit} and the initial clique calculation in the proof of Lemma \ref{lem:centering-bulk}, uniformly for $k \ne j$,
\[
\E\left[M_{k j}\right]
\ = \
F \left(\frac{k}{n}, \frac{j}{n}\right) + O \left(\frac{1}{n}\right)
\ = \
v^{\left(n\right)}_k v^{\left(n\right)}_j + O \left(\frac{1}{n}\right).
\]
Set
\[
R^{\left(n\right)}
\ := \
\E\left[M\right] - v^{\left(n\right)} \left(v^{\left(n\right)}\right)^\top.
\]
For $k \ne j$,
\[
R^{\left(n\right)}_{k j}
\ = \
O \left(\frac{1}{n}\right),
\]
while, since $M_{k k} = 0$,
\[
R^{\left(n\right)}_{k k}
\ = \
- \left(v^{\left(n\right)}_k\right)^2.
\]
Hence
\[
\norm{R^{\left(n\right)} - \operatorname{diag} \left(R^{\left(n\right)}_{11}, \dots, R^{\left(n\right)}_{nn}\right)}_F^2
\ = \
\sum_{k \ne j} O \left(\frac{1}{n^2}\right)
\ = \
O \left(1\right).
\]
Also,
\[
\norm{\operatorname{diag} \left(R^{\left(n\right)}_{11}, \dots, R^{\left(n\right)}_{nn}\right)}
\ = \
\max_{1 \le k \le n} \left(v^{\left(n\right)}_k\right)^2
\ \le \
1.
\]
Therefore
\[
\norm{R^{\left(n\right)}}
\ \le \
\norm{R^{\left(n\right)} - \operatorname{diag} \left(R^{\left(n\right)}_{11}, \dots, R^{\left(n\right)}_{nn}\right)}_F + \norm{\operatorname{diag} \left(R^{\left(n\right)}_{11}, \dots, R^{\left(n\right)}_{nn}\right)}
\ = \
O \left(1\right).
\]
Thus
\[
\norm{\E\left[M\right] - v^{\left(n\right)} \left(v^{\left(n\right)}\right)^\top}
\ = \
O \left(1\right).
\]

The rank--one matrix
\[
v^{\left(n\right)} \left(v^{\left(n\right)}\right)^\top
\]
has eigenvalues $\norm{v^{\left(n\right)}}_2^2$ and $0$ with multiplicity $n - 1$. Hence, by Proposition \ref{prop:weyl},
\[
\abs{\lambda_1 \left(\E\left[M\right]\right) - \norm{v^{\left(n\right)}}_2^2}
\ \le \
\norm{R^{\left(n\right)}}
\ = \
O \left(1\right),
\]
and, for $2 \le j \le n$,
\[
\abs{\lambda_j \left(\E\left[M\right]\right)}
\ \le \
\norm{R^{\left(n\right)}}
\ = \
O \left(1\right).
\]

It remains to compute $\norm{v^{\left(n\right)}}_2^2$. By definition,
\[
\norm{v^{\left(n\right)}}_2^2
\ = \
\sum_{k = 1}^n \phi \left(\frac{k}{n}\right)^2
\ = \
\sum_{k = 1}^{c n} 1 + \sum_{k = c n + 1}^{n} \frac{c}{\frac{2 k}{n} - c}.
\]
The first sum equals $c n + O \left(1\right)$. Write the second as a Riemann sum:
\[
\sum_{k = c n + 1}^{n} \frac{c}{\frac{2 k}{n} - c}
\ = \
n \int_c^1 \frac{c}{2 x - c} \, d x + O \left(1\right)
\ = \
n \cdot \frac{c}{2} \log \frac{2 - c}{c} + O \left(1\right).
\]
Therefore
\[
\norm{v^{\left(n\right)}}_2^2
\ = \
n \left(c + \frac{c}{2} \log \frac{2 - c}{c}\right) + O \left(1\right)
\ = \
\alpha \left(c\right) n + O \left(1\right).
\]
So
\[
\lambda_1 \left(\E\left[M\right]\right)
\ = \
\alpha \left(c\right) n + O \left(1\right),
\]
and
\[
\max_{2 \le j \le n} \abs{\lambda_j \left(\E\left[M\right]\right)}
\ = \
O \left(1\right),
\]
as claimed.
\end{proof}

\subsection{Outlier eigenvalue of $M$}

\begin{prop}[Asymptotic outlier eigenvalue]\label{prop:outlier-short}
Let $M$ be the adjacency matrix of the B--A multigraph with $m = c n$ and fixed $c \in \left(0, 1 / 2\right)$. Then
\[
\lambda_{\max} \left(M\right)
\ = \
\alpha \left(c\right) n + O_{\P} \left(1\right),
\]
where $\alpha \left(c\right)$ is given by~\eqref{eq:alpha-def}.
\end{prop}

\begin{proof}
Set
\[
W
\ := \
M - \E\left[M\right],
\qquad
Y
\ := \
\frac{W}{\sqrt{m}},
\qquad
m
\ = \
c n.
\]

We first derive a bound for separating the top eigenvalue from the rest of the spectrum. Let $H$ be the matrix from Proposition \ref{prop:independent-poisson-comparison}. Set
\[
s_{i j}^{\left(n\right)}
\ := \
\E\left[H_{i j}^2\right].
\]
From the proof of Proposition \ref{prop:check-sigma-conditions},
\[
s_{i j}^{\left(n\right)}
\ \le \
\frac{C_2}{n},
\qquad
\E\left[H_{i j}^4\right]
\ \le \
\frac{C_4}{n^2},
\]
with $C_2, C_4$ independent of $n, i, j$.

Expand
\[
\E\left[\Tr \left(H^4\right)\right]
\ = \
\sum_{i, j, k, \ell = 1}^n \E\left[H_{i j} H_{j k} H_{k \ell} H_{\ell i}\right].
\]

Since the entries are independent and centered, a term vanishes when some entry occurs only once. We have $H_{i i} = 0$, so the remaining terms have $i = k$ or $j = \ell$. Therefore
\[
\E\left[\Tr \left(H^4\right)\right]
\ = \
\sum_{i, j = 1}^n \E\left[H_{i j}^4\right] + 2 \sum_{i = 1}^n \sum_{\substack{j, k = 1 \\ j \ne k}}^n s_{i j}^{\left(n\right)} s_{i k}^{\left(n\right)}
\ \le \
\sum_{i, j = 1}^n \E\left[H_{i j}^4\right] + 2 \sum_{i = 1}^n \left(\sum_{j = 1}^n s_{i j}^{\left(n\right)}\right)^2
\ \le \
C_4 + 2 n C_2^2.
\]
By symmetry of $H$,
\[
\norm{H}^4
\ = \
\max_j \abs{\lambda_j \left(H\right)}^4
\ \le \
\sum_{j = 1}^n \lambda_j \left(H\right)^4
\ = \
\Tr \left(H^4\right).
\]
By Markov's inequality, for all $T > 0$,
\[
\P \left(\norm{H} > T n^{1 / 4}\right)
\ \le \
\frac{\E\left[\Tr \left(H^4\right)\right]}{T^4 n}
\ \le \
\frac{C_4 + 2 C_2^2}{T^4}.
\]
Therefore
\[
\norm{H}
\ = \
O_{\P} \left(n^{1 / 4}\right).
\]

From Proposition \ref{prop:deconditioning} we know that
\[
\E\left[\norm{\overline Y - \widetilde Y}_F^2\right]
\ = \
O \left(\sqrt{n}\right),
\]
and therefore
\[
\norm{\overline Y - \widetilde Y}_F
\ = \
O_{\P} \left(n^{1 / 4}\right),
\qquad
\overline Y
\ \stackrel{d}{=} \
Y.
\]
Similarly, from Proposition \ref{prop:independent-poisson-comparison} we know that
\[
\E\left[\norm{\widehat Y - H}_F^2\right]
\ = \
O \left(\sqrt{n}\right),
\]
so
\[
\norm{\widehat Y - H}_F
\ = \
O_{\P} \left(n^{1 / 4}\right),
\qquad
\widehat Y
\ \stackrel{d}{=} \
\widetilde Y.
\]
Since
\[
\norm{\widehat Y}
\ \le \
\norm{H} + \norm{\widehat Y - H}_F,
\]
we have
\[
\norm{\widetilde Y}
\ = \
O_{\P} \left(n^{1 / 4}\right).
\]
Similarly,
\[
\norm{Y}
\ = \
O_{\P} \left(n^{1 / 4}\right),
\]
which implies
\begin{equation}\label{eq:coarse-W-bound}
\norm{W}
\ = \
\sqrt{m} \norm{Y}
\ = \
O_{\P} \left(n^{3 / 4}\right)
\ = \
o_{\P} \left(n\right).
\end{equation}

By Lemma \ref{lem:top-eig-EM},
\[
\lambda_2 \left(\E\left[M\right]\right)
\ = \
O \left(1\right).
\]
Hence Proposition \ref{prop:weyl} and \eqref{eq:coarse-W-bound} give
\begin{equation}\label{eq:lambda2-M-small}
\lambda_2 \left(M\right)
\ = \
o_{\P} \left(n\right).
\end{equation}

Now define
\[
\widehat v^{\left(n\right)}
\ := \
\frac{v^{\left(n\right)}}{\norm{v^{\left(n\right)}}_2}.
\]
Since
\[
0
\ < \
v^{\left(n\right)}_i
\ \le \
1
\]
and, by Lemma \ref{lem:top-eig-EM},
\[
\norm{v^{\left(n\right)}}_2^2
\ = \
\alpha \left(c\right) n + O \left(1\right),
\]
we get
\begin{equation}\label{eq:vhat-infinity}
\norm{\widehat v^{\left(n\right)}}_{\infty}
\ = \
O \left(n^{-1 / 2}\right).
\end{equation}

For $k > t_0$, define
\[
\gamma_k
\ := \
\frac{m}{D_{k - 1}},
\]
so that
\[
\gamma_k
\ = \
O \left(\frac{1}{n}\right)
\]
uniformly in $k$. Define
\[
Z_k
\ := \
\sum_{j = 1}^k \widehat v_j^{\left(n\right)} d_j \left(k\right),
\qquad
 t_0 \ \le \ k \ \le \ n.
\]
and
\[
\xi_k
\ := \
\sum_{j < k} \widehat v^{\left(n\right)}_j \left(M_{k j} - m p_{k j}\right).
\]
We know from the proof of Proposition \ref{prop:deconditioning} that, conditional on $\mathcal F_{k - 1}$, the vector
\[
\left(M_{k 1}, \dots, M_{k, k - 1}\right)
\]
is multinomial. Therefore,
\[
\E\left[\xi_k \mid \mathcal F_{k - 1}\right]
\ = \
0,
\]
and
\begin{align*}
\E\left[\xi_k^2 \mid \mathcal F_{k - 1}\right]
&\ = \
m \left(\sum_{j < k} p_{k j} \left(\widehat v^{\left(n\right)}_j\right)^2 - \left(\sum_{j < k} p_{k j} \widehat v^{\left(n\right)}_j\right)^2\right) \\
&\ \le \
m \norm{\widehat v^{\left(n\right)}}_{\infty}^2
\ = \
O \left(1\right).
\end{align*}

The degree evolution gives
\[
Z_k
\ = \
\left(1 + \gamma_k\right) Z_{k - 1} + m \widehat v^{\left(n\right)}_k + \xi_k.
\]
Since $Z_{t_0}$ is deterministic and
\[
\E\left[\xi_k \mid \mathcal F_{k - 1}\right]
\ = \
0,
\]
the law of total variance gives
\[
\Var \left(Z_k\right)
\ \le \
\left(1 + \frac{C}{n}\right) \Var \left(Z_{k - 1}\right) + C.
\]
Iterating for at most $n$ steps gives
\begin{equation}\label{eq:Z-variance}
\Var \left(Z_k\right)
\ = \
O \left(n\right)
\end{equation}
uniformly in $k$.

For a real symmetric $n \times n$ matrix $A$ and a nonzero vector $x \in \R^n$, the Rayleigh quotient is defined by
\[
\mathcal R_A \left(x\right)
\ := \
\frac{x^\top A x}{x^\top x}.
\]
By \cite[Theorem 4.2.2]{HornJohnson2013},
\[
\lambda_{\min} \left(A\right)
\ \le \
\mathcal R_A \left(x\right)
\ \le \
\lambda_{\max} \left(A\right).
\]
Since $\norm{\widehat v^{\left(n\right)}}_2 = 1$, its Rayleigh quotient for $W$ is
\[
\mathcal R_W \left(\widehat v^{\left(n\right)}\right)
\ = \
\left(\widehat v^{\left(n\right)}\right)^\top W \widehat v^{\left(n\right)}.
\]

We now estimate the fluctuation of this Rayleigh quotient. Since the initial clique is deterministic,
\begin{align*}
\left(\widehat v^{\left(n\right)}\right)^\top W \widehat v^{\left(n\right)}
&\ = \
2 \sum_{k > t_0} \widehat v^{\left(n\right)}_k \xi_k + 2 \sum_{k > t_0} \widehat v^{\left(n\right)}_k \gamma_k \left(Z_{k - 1} - \E\left[Z_{k - 1}\right]\right).
\end{align*}
For $t_0 < \ell < k$, the random variable $\xi_\ell$ is $\mathcal F_{k - 1}$-measurable, with
\[
\E\left[\xi_k \mid \mathcal F_{k - 1}\right]
\ = \
0.
\]
Therefore, by the tower property,
\begin{align*}
\E\left[\xi_\ell \xi_k\right]
&\ = \
\E\left[\E\left[\xi_\ell \xi_k \mid \mathcal F_{k - 1}\right]\right] \\
&\ = \
\E\left[\xi_\ell \E\left[\xi_k \mid \mathcal F_{k - 1}\right]\right] \\
&\ = \
0.
\end{align*}
Hence all cross terms vanish, so
\begin{align*}
\E\left[\abs{\sum_{k > t_0} \widehat v^{\left(n\right)}_k \xi_k}^2\right]
&\ = \
\sum_{k > t_0} \left(\widehat v^{\left(n\right)}_k\right)^2 \E\left[\xi_k^2\right] \\
&\ \le \
C \sum_{k > t_0} \left(\widehat v^{\left(n\right)}_k\right)^2 \\
&\ \le \
C \norm{\widehat v^{\left(n\right)}}_2^2 \\
&\ = \
C.
\end{align*}
Therefore
\[
\E\left[\abs{\sum_{k > t_0} \widehat v^{\left(n\right)}_k \xi_k}^2\right]
\ = \
O \left(1\right).
\]

For the second term, Minkowski's inequality, \eqref{eq:vhat-infinity}, and \eqref{eq:Z-variance} give
\begin{align*}
\norm{\sum_{k > t_0} \widehat v^{\left(n\right)}_k \gamma_k \left(Z_{k - 1} - \E\left[Z_{k - 1}\right]\right)}_{L^2}
&\ \le \
\sum_{k > t_0} \abs{\widehat v^{\left(n\right)}_k} \gamma_k \sqrt{\Var \left(Z_{k - 1}\right)} \\
&\ = \
O \left(1\right).
\end{align*}
Therefore
\begin{equation}\label{eq:rayleigh-fluctuation}
\left(\widehat v^{\left(n\right)}\right)^\top W \widehat v^{\left(n\right)}
\ = \
O_{\P} \left(1\right).
\end{equation}

We next prove
\begin{equation}\label{eq:directional-W-bound}
\norm{W \widehat v^{\left(n\right)}}
\ = \
O_{\P} \left(\sqrt{n}\right).
\end{equation}
For every $i$, let
\[
X_i
\ := \
\sum_{j = 1}^n \widehat v^{\left(n\right)}_j \left(M_{i j} - \E\left[M_{i j}\right]\right).
\]
We show that
\[
\E\left[X_i^2\right]
\ = \
O \left(1\right)
\]
uniformly in $i$.

If $i > t_0$,
\[
\sum_{j < i} \widehat v^{\left(n\right)}_j \left(M_{i j} - \E\left[M_{i j}\right]\right)
\ = \
\xi_i + \gamma_i \left(Z_{i - 1} - \E\left[Z_{i - 1}\right]\right),
\]
whose $L^2$ norm is $O \left(1\right)$ by the preceding estimates. If $i \le t_0$, the contribution of the initial clique is deterministic, so it vanishes after centering.

For $k > \max\left\{i, t_0\right\}$, define
\[
\eta_{k i}
\ := \
M_{k i} - \gamma_k d_i \left(k - 1\right).
\]
Then
\[
\E\left[\eta_{k i} \mid \mathcal F_{k - 1}\right]
\ = \
0,
\]
and, since $M_{k i}$ is conditionally binomial,
\[
\E\left[\eta_{k i}^2 \mid \mathcal F_{k - 1}\right]
\ \le \
\gamma_k d_i \left(k - 1\right).
\]
The uniform bound
\[
\E\left[d_i \left(k - 1\right)\right]
\ = \
O \left(n\right)
\]
then gives
\[
\E\left[\eta_{k i}^2\right]
\ = \
O \left(1\right).
\]
For $\max\left\{i, t_0\right\} < \ell < k$, the random variable $\eta_{\ell i}$ is $\mathcal F_{k - 1}$-measurable, with
\[
\E\left[\eta_{k i} \mid \mathcal F_{k - 1}\right]
\ = \
0.
\]
Therefore, by the tower property,
\begin{align*}
\E\left[\eta_{\ell i} \eta_{k i}\right]
&\ = \
\E\left[\E\left[\eta_{\ell i} \eta_{k i} \mid \mathcal F_{k - 1}\right]\right] \\
&\ = \
\E\left[\eta_{\ell i} \E\left[\eta_{k i} \mid \mathcal F_{k - 1}\right]\right] \\
&\ = \
0.
\end{align*}
Hence all cross terms vanish, and therefore
\begin{align*}
\E\left[\abs{\sum_{k > \max\left\{i, t_0\right\}} \widehat v^{\left(n\right)}_k \eta_{k i}}^2\right]
&\ = \
\sum_{k > \max\left\{i, t_0\right\}} \left(\widehat v^{\left(n\right)}_k\right)^2 \E\left[\eta_{k i}^2\right] \\
&\ \le \
C \sum_{k > \max\left\{i, t_0\right\}} \left(\widehat v^{\left(n\right)}_k\right)^2 \\
&\ \le \
C \norm{\widehat v^{\left(n\right)}}_2^2 \\
&\ = \
C.
\end{align*}
Thus
\[
\norm{\sum_{k > \max\left\{i, t_0\right\}} \widehat v^{\left(n\right)}_k \eta_{k i}}_{L^2}
\ = \
O \left(1\right).
\]

For the remaining part, using
\[
\Var \left(d_i \left(k - 1\right)\right)
\ = \
O \left(n\right)
\]
from the proof of Proposition \ref{prop:deconditioning}, we get
\begin{align*}
\norm{\sum_{k > \max\left\{i, t_0\right\}} \widehat v^{\left(n\right)}_k \gamma_k \left(d_i \left(k - 1\right) - \E\left[d_i \left(k - 1\right)\right]\right)}_{L^2}
&\ \le \
\sum_{k > \max\left\{i, t_0\right\}} \abs{\widehat v^{\left(n\right)}_k} \gamma_k \sqrt{\Var \left(d_i \left(k - 1\right)\right)} \\
&\ = \
O \left(1\right).
\end{align*}
Therefore
\[
\E\left[X_i^2\right]
\ = \
O \left(1\right)
\]
uniformly in $i$. It follows that
\[
\E\left[\norm{W \widehat v^{\left(n\right)}}_2^2\right]
\ = \
\sum_{i = 1}^n \E\left[X_i^2\right]
\ = \
O \left(n\right),
\]
which proves \eqref{eq:directional-W-bound}.

Define
\[
\theta_n
\ := \
\left(\widehat v^{\left(n\right)}\right)^\top M \widehat v^{\left(n\right)}.
\]
By Lemma \ref{lem:top-eig-EM},
\[
\E\left[M\right]
\ = \
v^{\left(n\right)} \left(v^{\left(n\right)}\right)^\top + R^{\left(n\right)},
\qquad
\norm{R^{\left(n\right)}}
\ = \
O \left(1\right).
\]
Therefore, using \eqref{eq:rayleigh-fluctuation},
\begin{align*}
\theta_n
&\ = \
\norm{v^{\left(n\right)}}_2^2 + \left(\widehat v^{\left(n\right)}\right)^\top R^{\left(n\right)} \widehat v^{\left(n\right)} + \left(\widehat v^{\left(n\right)}\right)^\top W \widehat v^{\left(n\right)} \\
&\ = \
\alpha \left(c\right) n + O_{\P} \left(1\right).
\end{align*}
In particular, together with \eqref{eq:lambda2-M-small},
\begin{equation}\label{eq:outlier-gap}
\theta_n - \lambda_2 \left(M\right)
\ = \
\alpha \left(c\right) n + o_{\P} \left(n\right).
\end{equation}

Define the residual
\[
\varepsilon_n
\ := \
\norm{\left(M - \theta_n I\right) \widehat v^{\left(n\right)}}.
\]
Using
\[
M
\ = \
v^{\left(n\right)} \left(v^{\left(n\right)}\right)^\top + R^{\left(n\right)} + W,
\]
we obtain
\begin{align*}
\left(M - \theta_n I\right) \widehat v^{\left(n\right)}
&\ = \
W \widehat v^{\left(n\right)} - \left(\left(\widehat v^{\left(n\right)}\right)^\top W \widehat v^{\left(n\right)}\right) \widehat v^{\left(n\right)} + R^{\left(n\right)} \widehat v^{\left(n\right)} - \left(\left(\widehat v^{\left(n\right)}\right)^\top R^{\left(n\right)} \widehat v^{\left(n\right)}\right) \widehat v^{\left(n\right)}.
\end{align*}
Hence, by \eqref{eq:rayleigh-fluctuation}, \eqref{eq:directional-W-bound}, and $\norm{R^{\left(n\right)}} = O \left(1\right)$,
\begin{equation}\label{eq:residual-bound}
\varepsilon_n
\ = \
O_{\P} \left(\sqrt{n}\right).
\end{equation}

We now connect the Rayleigh quotient $\theta_n$ to the largest eigenvalue of $M$. Let
\[
\lambda_1 \left(M\right)
\ \ge \
\lambda_2 \left(M\right)
\ \ge \
\cdots
\ \ge \
\lambda_n \left(M\right)
\]
and let $u_1, \dots, u_n$ be an orthonormal eigenbasis. Write
\[
q_j
\ := \
\abs{\left\langle\widehat v^{\left(n\right)}, u_j\right\rangle}^2.
\]
On the event $\theta_n > \lambda_2 \left(M\right)$, which has probability approaching one by \eqref{eq:outlier-gap}, set
\[
\delta_n
\ := \
\lambda_1 \left(M\right) - \theta_n,
\qquad
g_n
\ := \
\theta_n - \lambda_2 \left(M\right).
\]
Since $\theta_n$ is a Rayleigh quotient,
\[
\delta_n
\ \ge \
0.
\]
Moreover,
\[
q_1 \delta_n
\ = \
\sum_{j = 2}^n q_j \left(\theta_n - \lambda_j \left(M\right)\right).
\]
Therefore
\begin{align*}
\varepsilon_n^2
&\ = \
q_1 \delta_n^2 + \sum_{j = 2}^n q_j \left(\theta_n - \lambda_j \left(M\right)\right)^2 \\
&\ \ge \
q_1 \delta_n^2 + g_n \sum_{j = 2}^n q_j \left(\theta_n - \lambda_j \left(M\right)\right) \\
&\ = \
q_1 \delta_n \left(\delta_n + g_n\right).
\end{align*}
Also,
\[
\theta_n
\ \le \
q_1 \lambda_1 \left(M\right) + \left(1 - q_1\right) \lambda_2 \left(M\right),
\]
so
\[
g_n
\ \le \
q_1 \left(\delta_n + g_n\right).
\]
Combining the last two inequalities gives
\[
\varepsilon_n^2
\ \ge \
\delta_n g_n,
\]
and hence
\[
0
\ \le \
\lambda_{\max} \left(M\right) - \theta_n
\ \le \
\frac{\varepsilon_n^2}{\theta_n - \lambda_2 \left(M\right)}.
\]
By \eqref{eq:outlier-gap} and \eqref{eq:residual-bound},
\[
\frac{\varepsilon_n^2}{\theta_n - \lambda_2 \left(M\right)}
\ = \
O_{\P} \left(1\right).
\]
Therefore
\[
\lambda_{\max} \left(M\right)
\ = \
\theta_n + O_{\P} \left(1\right)
\ = \
\alpha \left(c\right) n + O_{\P} \left(1\right),
\]
as claimed.
\end{proof}

Theorem \ref{thm:top-eigenvalue} is a corollary of Proposition \ref{prop:outlier-short}.

\section{AI statement}
We used ChatGPT-6 Astra to proof-read an earlier draft of the paper. In particular, it suggested to read \cite{Zhu2018AGA} and use it instead of our previous approach that relied on \cite{AEK2017} directly. We have reviewed all AI-assisted work.

In addition, Aristotle by Harmonic was used to generate a Lean formalization of the paper.

\bibliographystyle{plain}

\bibliography{bib}

\end{document}